\documentclass[a4paper,reqno]{amsart}
\usepackage[LGR,T1]{fontenc}
\usepackage[utf8]{inputenc}
\usepackage{xcolor}
\definecolor{lime}{rgb}{0.652344, 0.808594, 0.226562}
\usepackage{prettyref}
\usepackage{mathtools}
\usepackage{enumitem}
\usepackage{amstext}
\usepackage{amsthm}
\usepackage{amssymb}
\usepackage{stmaryrd}
\PassOptionsToPackage{normalem}{ulem}
\usepackage{ulem}
\usepackage{microtype}
\usepackage[bookmarks=false,
 breaklinks=true,pdfborder={0 0 1},backref=false,colorlinks=true]
 {hyperref}
\hypersetup{
 pdfborderstyle=,linkcolor=black,citecolor=black,urlcolor=black}

\makeatletter

\providecolor{lyxadded}{rgb}{0,0,1}
\providecolor{lyxdeleted}{rgb}{1,0,0}
\DeclareRobustCommand{\mklyxadded}[1]{\textcolor{lyxadded}\bgroup#1\egroup}
\DeclareRobustCommand{\mklyxdeleted}[1]{\textcolor{lyxdeleted}\bgroup\mklyxsout{#1}\egroup}
\DeclareRobustCommand{\mklyxsout}[1]{\ifx\\#1\else\sout{#1}\fi}

\numberwithin{equation}{section}
\numberwithin{figure}{section}
\newcommand\thmsname{\protect\theoremname}
\newcommand\nm@thmtype{theorem}
\newenvironment{namedthm}[1][Undefined Theorem Name]{
  \ifx{#1}{Undefined Theorem Name}\renewcommand\nm@thmtype{theorem*}
  \else\renewcommand\thmsname{#1}\renewcommand\nm@thmtype{namedtheorem}
  \fi
  \begin{\nm@thmtype}}
  {\end{\nm@thmtype}}

\usepackage{pgfplots}
\usepgfplotslibrary{colormaps,fillbetween}
\pgfplotsset{compat=1.15}
\usetikzlibrary{arrows}
\usepackage{xcolor}
\usepackage{prettyref}
\usepackage{mathrsfs}
\usepackage{amstext}
\usepackage{amsthm}
\usepackage{bbm}
\usepackage{txfonts}
\usepackage{textcomp}
\usepackage{dsfont}
\usepackage{graphicx}
\usepackage{caption}
\usepackage{subcaption}
\usepackage{float}
\usepackage{tikz}
\usepackage{pgfplots}
\usepgfplotslibrary{colormaps,fillbetween}
\pgfplotsset{compat=1.15}
\usetikzlibrary{arrows}
\usepackage[foot]{amsaddr}
\usepackage{url}

\usepackage[capitalise,noabbrev]{cleveref}

\numberwithin{equation}{section}
\numberwithin{figure}{section}

\newcommand{\N}{\mathbb N}
\newcommand{\Np}{\mathbb N_{\geq1}}
\newcommand{\R}{\mathbb R}
\newcommand{\D}{\mathcal D}

\newcommand{\Q}{\mathfrak Q}
\newcommand{\J}{\mathfrak J}

\newcommand{\1}{\mathbbm 1}

\renewcommand{\d}{\;\mathrm d}
\renewcommand{\phi}{\varphi}
\renewcommand{\epsilon}{\varepsilon}
\renewcommand{\emptyset}{\varnothing}
\renewcommand{\rho}{\varrho}

\DeclareMathOperator{\dom}{dom}

\DeclareMathOperator{\supp}{supp}
\DeclareMathOperator{\card}{card}
\DeclareMathOperator{\diam}{diam}

\newcommand{\ED}{\mathcal E^D}
\newcommand{\EN}{\mathcal E^N}
\newcommand{\DeltaD}{\Delta_\nu^D}
\newcommand{\DeltaN}{\Delta_\nu^N}

\theoremstyle{plain}
\newtheorem{thm}{Theorem}[section]
\newtheorem{prop}[thm]{Proposition}\newtheorem{cor}[thm]{Corollary}\newtheorem{lem}[thm]{Lemma}

\theoremstyle{definition}
\newtheorem{defn}[thm]{Definition}\newtheorem{example}[thm]{Example}

\theoremstyle{remark}
\newtheorem{rem}[thm]{Remark}

\crefname{thm}{theorem}{theorems}
\Crefname{thm}{Theorem}{Theorems}
\crefname{prop}{proposition}{propositions}
\Crefname{prop}{Proposition}{Propositions}
\crefname{cor}{corollary}{corollaries}
\Crefname{cor}{Corollary}{Corollaries}
\crefname{lem}{lemma}{lemmas}
\Crefname{lem}{Lemma}{Lemmas}
\crefname{defn}{definition}{definitions}
\Crefname{defn}{Definition}{Definitions}
\crefname{example}{example}{examples}
\Crefname{example}{Example}{Examples}
\crefname{rem}{remark}{remarks}
\Crefname{rem}{Remark}{Remarks}

\AtBeginDocument{
\DeclareFieldFormat[article]{title}{{#1}}
\DeclareFieldFormat[phdthesis]{citetitle}{{\em #1}}
\DeclareFieldFormat[unpublished]{title}{{#1}}
\DeclareFieldFormat[proceedings]{title}{{#1}}
\DeclareFieldFormat[incollection]{title}{{#1}}
\DeclareFieldFormat[unpublished]{title}{{\em #1}}
\renewbibmacro{in:}{}
\DeclareFieldFormat[article]{pages}{#1}
}

\newrefformat{thmA}{Theorem A}
\newrefformat{thmB}{Theorem B}
\newrefformat{thmC}{Theorem C}
\newrefformat{thmD}{Theorem D}
\newrefformat{thmE}{Theorem E}
\newrefformat{prop}{Proposition~\ref{#1}}
\newrefformat{rem}{Remark~\ref{#1}}
\newrefformat{fact}{Fact~\ref{#1}}
\newrefformat{exa}{Example~\ref{#1}}
\newrefformat{cor}{Corollary~\ref{#1}}
\newrefformat{assu}{Assumption~\ref{#1}}
\newrefformat{subsec}{Section~\ref{#1}}
\newrefformat{sec}{Section~\ref{#1}}
\newrefformat{def}{Definition~\ref{#1}}

\DeclareRobustCommand{\orcidicon}{%
  \begin{tikzpicture}
    \draw[lime,fill=lime] (0,0)
      circle[radius=0.16]
      node[white] {{\fontfamily{qag}\selectfont\tiny ID}};
    \draw[white,fill=white] (-0.0625,0.095)
      circle[radius=0.007];
  \end{tikzpicture}
  \hspace{-2mm}%
}

\newcommand{\orcidauthorA}{0000-0003-3156-4522}
\newcommand{\orcidauthorB}{0000-0003-4484-7919}

\newcommand{\orcidA}{%
  \href{https://orcid.org/\orcidauthorA}{\orcidicon}}
\newcommand{\orcidB}{%
  \href{https://orcid.org/\orcidauthorB}{\orcidicon}}

\makeatother

\theoremstyle{plain}

\usepackage[style=alphabetic,backend=bibtex,citestyle=alphabetic,backref=false,giveninits=true,url=false,isbn=false,doi=true,eprint=false,maxbibnames=99]{biblatex}

\begin{document}
\title[Spectral dimensions of Kre\u{\i}n--Feller operators]{Spectral dimensions of Kre\u{\i}n--Feller operators in higher dimensions}
\author{Marc Kesseböhmer\orcidA{}}
\email{mhk@uni-bremen.de}
\author{Aljoscha Niemann\orcidB{}}
\email{niemann1@uni-bremen.de}
\address{Faculty 3 -- Mathematics and Computer Science, University of Bremen,
Bibliothekstr.\ 5, 28359 Bremen, Germany}
\date{\today}
\keywords{Kre\u{\i}n--Feller operator; Laplace operator; spectral asymptotics; $L^{q}$-spectrum;
spectral partition function; Dirichlet forms; Minkowski dimension;
coarse multifractal formalism; Sobolev spaces; adaptive approximation}
\subjclass[2000]{Primary 35P20, 35J05; Secondary 28A80, 42B35, 45D05}
\thanks{This research was supported by the DFG grant KE 1440/3-1.}
\begin{abstract}
We study the Dirichlet and Neumann spectral dimensions of Kreĭn--Feller
operators associated with finite non-zero Borel measures $\nu$ supported
in the open $d$-dimensional unit cube, $d\geq2$. The central new
ingredient is a monotone dyadic spectral set function $\J_{\nu,a,b}$,
whose value on a cube is defined by taking a supremum over all of
its dyadic subcubes. This supremum captures the worst local measure--Sobolev
embedding scale. The resulting spectral partition function provides
a higher-dimensional extension of the $L^{q}$-spectrum, which governs
the one-dimensional theory developed in our earlier work. If the lower
$L^{\infty}$-dimension of $\nu$ is larger than $d-2$, the relevant
measure--Sobolev embeddings are compact and the upper Dirichlet and
Neumann spectral dimensions coincide with the zero of the spectral
partition function. Under an explicit regularity condition on the
partition function, both spectral dimensions exist and agree. Below
the threshold the form construction fails; at the critical threshold,
examples exhibit failure of continuity, continuity without compactness,
and compactness with an existing spectral dimension. Applications
include absolutely continuous measures, Ahlfors--David regular measures,
and self-conformal measures without a separation assumption. In the
self-conformal setting the spectral partition function, and hence
the spectral dimension, is determined by the $L^{q}$-spectrum. We
also derive two-sided eigenvalue-counting bounds for Ahlfors--David
regular measures and construct an example for which the spectral dimension
does not exist.
\end{abstract}

\maketitle

\section{Introduction and statement of the main results}

\subsection{Introduction and background}

\label{subsec:Introduction-and-background}

In this article, we extend our work on spectral dimensions of Kreĭn--Feller
operators associated with compactly supported finite Borel measures
to higher dimensions. Kreĭn--Feller operators in one dimension originate
in \cite{MR0042045,Fe57,KK68}. The subject has since developed extensively;
see, for example, \cite{MR0107037,MR118891,MR146444,MR0278126,MR661628,Fu87,MR1328700,MR2135259,MR2828537,A15b,MR3318648,MR3648085,MR3809018,MR4176086}
and the references therein. More recently, in \cite{KN21,KN2022},
the authors gave an almost complete description of the relationship
between the spectral dimension and the $L^{q}$-spectrum of the underlying
measure.

For ambient dimension $d\geq2$, the situation is substantially more
delicate. First steps were taken in \cite{MR2261337} and \cites{MR1484417}{MR1999566}{MR2087139}.
The connection between fractal geometry and the spectral properties
of naturally associated operators is a long-standing theme; see \cite{MR1021743,MR1026205,MR1126694,MR1243717,MR1473565}.
For connections with the Weyl--Berry conjecture and the work of Lapidus
and collaborators, see also \prettyref{rem:CanOneHearFractalDrum}.
Recent extensions of the embedding and form theory to Kreĭn--Feller
operators on Riemannian manifolds are developed in \cite{NgaiOuyang2026}.

We introduce a spectral partition function which generalises the $L^{q}$-spectrum
and is adapted to higher dimensions. This construction also clarifies
certain optimal embedding constants for Sobolev spaces, as studied
by Maz'ya and Preobrazhenskii \cite{MR817985,MR743823} in dimension
$2$ and by Adams \cite{MR336316}; see also \cite[Sections~1.4.1--1.4.2]{MR2777530}
in dimensions $d>2$.

Previous higher-dimensional spectral asymptotics cover several important
special situations. Naimark and Solomyak treated self-similar measures
satisfying the open set condition in \cite{MR1298682,MR1338787}.
Triebel considered, in particular, Ahlfors--David regular measures
\cite[Theorem 30.2]{MR1484417}; see also \cite{MR4331823,MR4484835,Ngai_2021}.
Ngai and Xie asked in \cite[Section 5]{Ngai_2021} about self-similar
measures with overlaps in dimensions $d\geq2$. As an application
of our general results, we treat self-conformal measures without a
separation assumption and identify the spectral dimension with the
unique intersection of the $L^{q}$-spectrum and the relevant affine
function.

This work is based in part on the dissertation \cite{elib_6573} of
the second author.

The principal conclusions are summarised in Theorems A--E below.
\prettyref{thmA:MainUpperBound_General} and \prettyref{thmB:IntroGeneralPrincipleLowerBound}
give the abstract upper and lower spectral bounds; \prettyref{thmC:MainChain}
relates the spectral dimensions to the partition quantities; \prettyref{thmD:PF-regular}
identifies the common spectral dimension with the zero of the spectral
partition function under an explicit regularity hypothesis; and \prettyref{thmE:Self-Conformal}
treats self-conformal measures without a separation assumption. To
state these results precisely, we first introduce the reduced form
domains, eigenvalue-counting functions, and geometric notation on
which they depend. The standard form-theoretic proofs are deferred
to the appendix.

\subsection{Sobolev spaces and embeddings into Lebesgue spaces}

\label{subsec:Sobolev-canonical-embedding}

Let $\Omega\subset\R^{d}$ be a bounded connected Lipschitz domain,
and let $\nu$ be a finite non-zero Borel measure satisfying $\supp\nu\subset\Omega.$

We write $H^{1}(\Omega)\coloneqq W^{1,2}(\Omega)$ and equip it with
the inner product 
\[
\langle f,g\rangle_{H^{1}(\Omega)}\coloneqq\int_{\Omega}fg\d\Lambda+\int_{\Omega}\nabla f\cdot\nabla g\d\Lambda.
\]
 Furthermore, $H^{1}_{0}(\Omega)\coloneqq\overline{\mathcal{C}^{\infty}_{c}(\Omega)}^{\,H^{1}(\Omega)}.$
The \emph{Poincaré inequality} for $f\in H^{1}_{0}(\Omega)$, and
the \emph{Poincaré--Wirtinger inequality} for $u\in H^{1}(\Omega)$
with $\int_{\Omega}u\d\Lambda=0,$ give
\begin{equation}
\left\Vert u\right\Vert _{L^{2}(\Omega)}\leq C_{\Omega}\left\Vert \nabla u\right\Vert _{L^{2}(\Omega)}.\label{eq:PW_inequality}
\end{equation}
Consequently, $\langle f,g\rangle_{H^{1}_{0}(\Omega)}\coloneqq\int_{\Omega}\nabla f\cdot\nabla g\d\Lambda$
is an inner product on $H^{1}_{0}(\Omega)$, and its norm is equivalent
to the $H^{1}(\Omega)$-norm.

For $1\leq p\leq\infty$, we write $L^{p}(\Omega)\coloneqq L^{p}(\Omega,\Lambda)$
for Lebesgue measure and $L^{p}(\Omega,\nu)$ for the measure $\nu$.
When the domain is clear, we abbreviate this to $L^{p}(\nu)$.

Let $\iota_{0}:\mathcal{C}^{\infty}(\overline{\Omega})\longrightarrow L^{2}(\Omega,\nu)$
denote pointwise restriction. Assume that there is a constant $C>0$
such that $\|\iota_{0}f\|_{L^{2}(\Omega,\nu)}\leq C\|f\|_{H^{1}(\Omega)}$
for every $f\in\mathcal{C}^{\infty}(\overline{\Omega})$. Then $\iota_{0}$
has a unique bounded extension
\[
\iota:H^{1}(\Omega)\longrightarrow L^{2}(\Omega,\nu),
\]
 whose restriction to $H^{1}_{0}(\Omega)$ will again be denoted by
$\iota$.

\subsection{Dirichlet and Neumann forms}

\label{subsec:augmented-Sobolev-inner-product}

The proofs of the statements in this section are mostly standard and
are postponed to the appendix.

Define 
\[
\langle f,g\rangle_{\iota}\coloneqq\int_{\Omega}\nabla f\cdot\nabla g\d\Lambda+\langle\iota f,\iota g\rangle_{L^{2}(\Omega,\nu)},\qquad f,g\in H^{1}(\Omega).
\]
 The associated norm is $\|f\|^{2}_{\iota}\coloneqq\|\nabla f\|^{2}_{L^{2}(\Omega)}+\|\iota f\|^{2}_{L^{2}(\Omega,\nu)}.$
\begin{lem}
\label{lem:augmented-norm-equivalence} The form $\langle\cdot,\cdot\rangle_{\iota}$
is an inner product on $H^{1}(\Omega)$, and $\|\cdot\|_{\iota}$
is equivalent to the standard $H^{1}(\Omega)$-norm.
\end{lem}

Define
\[
\mathcal{N}^{N}\coloneqq\ker\bigl(\iota|_{H^{1}(\Omega)}\bigr)\qquad\text{and }\qquad\mathcal{N}^{D}\coloneqq\ker\bigl(\iota|_{H^{1}_{0}(\Omega)}\bigr).
\]
 The kernels $\mathcal{N}^{N}$ and $\mathcal{N}^{D}$ are closed
in $H^{1}(\Omega)$ and $H^{1}_{0}(\Omega)$, respectively, with respect
to $\langle\cdot,\cdot\rangle_{\iota}$. Define the \emph{reduced
Sobolev spaces} by
\begin{align*}
H^{N}_{\nu} & \coloneqq\bigl\{ f\in H^{1}(\Omega):\langle f,g\rangle_{\iota}=0\ \text{for every }g\in\mathcal{N}^{N}\bigr\},\\
H^{D}_{\nu} & \coloneqq\bigl\{ f\in H^{1}_{0}(\Omega):\langle f,g\rangle_{\iota}=0\ \text{for every }g\in\mathcal{N}^{D}\bigr\}.
\end{align*}
 This leads to the orthogonal decompositions with respect to $\langle\cdot,\cdot\rangle_{\iota}$,
\[
H^{1}(\Omega)=H^{N}_{\nu}\oplus\mathcal{N}^{N}\qquad\text{and }\,\,\,\,\,H^{1}_{0}(\Omega)=H^{D}_{\nu}\oplus\mathcal{N}^{D}.
\]
From now on $B$ will always belong to $\{D,N\}$ and will indicate
the Dirichlet or Neumann boundary condition. Define the injective
embedding corresponding to Dirichlet and Neumann boundary conditions
by
\[
\iota^{B}\coloneqq\iota|_{H^{B}_{\nu}}:H^{B}_{\nu}\longrightarrow L^{2}(\Omega,\nu).
\]
If $f$ from the full Sobolev space has the decomposition $f=f^{\perp}+f^{0}$,
with $f^{\perp}\in H^{B}_{\nu}$ and $f^{0}\in\mathcal{N}^{B}$, then
$\iota f=\iota f^{\perp}$. Consequently, $\iota^{N}(H^{N}_{\nu})=\iota(H^{1}(\Omega))$
and $\iota^{D}(H^{D}_{\nu})=\iota(H^{1}_{0}(\Omega))$.

Define the form domain by
\[
\dom(\mathcal{E}^{B})\coloneqq\iota^{B}(H^{B}_{\nu})=\begin{cases}
\iota\left(H^{1}(\Omega)\right), & B=N,\\
\iota\left(H^{1}_{0}(\Omega)\right), & B=D.
\end{cases}
\]

For $u,v\in\dom(\mathcal{E}^{B})$, define the corresponding Dirichlet
or Neumann form by
\[
\mathcal{E}^{B}(u,v)\coloneqq\int_{\Omega}\nabla(\iota^{B})^{-1}u\cdot\nabla(\iota^{B})^{-1}v\d\Lambda.
\]

\begin{lem}[\emph{Energy-minimising representative}]
 \label{lem:Minimizer}Let $u\in\dom(\mathcal{E}^{B})$. Then $(\iota^{B})^{-1}u$
is the unique $B$-admissible Sobolev representative of $u$ that
minimises the gradient energy. More precisely, 
\begin{align*}
\mathcal{E}^{N}(u,u) & =\min\left\{ \int_{\Omega}|\nabla f|^{2}\d\Lambda:f\in H^{1}(\Omega),\ \iota f=u\right\} ,\\
\mathcal{E}^{D}(u,u) & =\min\left\{ \int_{\Omega}|\nabla f|^{2}\d\Lambda:f\in H^{1}_{0}(\Omega),\ \iota f=u\right\} .
\end{align*}
In either case, the unique minimiser is $f=(\iota^{B})^{-1}u$.
\end{lem}

\begin{lem}[Neumann zero mode]
\label{lem:Neumann-zero-mode} Every constant function belongs to
$H^{N}_{\nu}$, and its image under $\iota^{N}$ has zero Neumann
energy. If $\Omega$ is connected, the kernel of the Neumann form
consists precisely of the constant $L^{2}(\Omega,\nu)$-functions.
\end{lem}

\begin{prop}[Density and closedness]
\label{prop:forms-density-closedness} Assume $\supp\nu\subset\Omega.$
Then $\mathcal{E}^{B}$ is a densely defined, non-negative, closed
symmetric form on $L^{2}(\Omega,\nu)$.
\end{prop}

The following proposition is included for completeness; it is not
used in the spectral arguments.
\begin{prop}
\label{prop:forms-Markovian} The form $\mathcal{E}^{B}$ is Markovian.
\end{prop}

\subsection{Associated operators}

\label{subsec:associated-operators}

By the first representation theorem for closed forms \cite[Chapter~VI, Theorem~2.1]{MR0407617}{]},
there exists a unique

non-negative self-adjoint operator $\Delta^{B}_{\nu}$ on $L^{2}(\Omega,\nu)$
such that 
\[
\mathcal{E}^{B}(u,v)=\langle\Delta^{B}_{\nu}u,v\rangle_{L^{2}(\Omega,\nu)},\qquad u\in\dom(\Delta^{B}_{\nu}),\quad v\in\dom(\mathcal{E}^{B}).
\]
 These are the \emph{Dirichlet} and \emph{Neumann Kreĭn--Feller operator}s.

If $\dom(\mathcal{E}^{B})\hookrightarrow L^{2}(\Omega,\nu)$ is compact,
where the form domain carries the norm $(\mathcal{E}^{B}(u,u)+\|u\|^{2}_{L^{2}(\nu)})^{1/2}$,
then $\Delta^{B}_{\nu}$ has compact resolvent. Its eigenvalues, repeated
according to multiplicity, may be written in non-decreasing order
as $(\lambda^{B}_{n})_{n\in\Np}$, with a corresponding orthonormal
basis $(\phi^{B}_{n})_{n\in\Np}$ of $L^{2}(\Omega,\nu)$. Since $\Omega$
is connected, $0=\lambda^{N}_{1}<\lambda^{N}_{2}\leq\cdots$ and the
first Neumann eigenspace consists of the constants, whereas $0<\lambda^{D}_{1}\leq\lambda^{D}_{2}\leq\cdots$.
In both cases, $\lambda^{B}_{n}\to\infty$.

Furthermore, the spectral theorem gives 
\begin{align*}
\dom(\mathcal{E}^{B}) & =\left\{ f=\sum_{n\in\Np}a_{n}\phi^{B}_{n}:\sum_{n\in\Np}\lambda^{B}_{n}|a_{n}|^{2}<\infty\right\} ,\\
\dom(\Delta^{B}_{\nu}) & =\left\{ f=\sum_{n\in\Np}a_{n}\phi^{B}_{n}:\sum_{n\in\Np}\left(\lambda^{B}_{n}\right)^{2}|a_{n}|^{2}<\infty\right\} .
\end{align*}

\subsection{Rayleigh quotient}

\label{subsec:unified-Rayleigh-quotient}

For $f\in H^{1}\left(\Omega\right)$, define the \emph{Rayleigh quotient
}
\[
R_{\nu}(f)\coloneqq\begin{cases}
{\displaystyle \frac{\int_{\Omega}|\nabla f|^{2}\d\Lambda}{\|\iota f\|^{2}_{L^{2}(\Omega,\nu)}}}, & \iota f\neq0,\\
+\infty, & \iota f=0.
\end{cases}
\]
The Dirichlet and Neumann cases differ only through the admissible
trial spaces: $H^{1}_{0}(\Omega)$ and $H^{1}(\Omega)$.
\begin{prop}[Unified min--max principle]
\label{prop:unified-min-max} Enumerate both eigenvalue sequences
by $\Np$, with $\lambda^{B}_{1}\leq\lambda^{B}_{2}\leq\cdots$. Then,
for every $n\in\Np$, 
\[
\lambda^{N}_{n}=\inf_{\substack{G\subset H^{1}(\Omega)\\
\dim G=n
}
}\ \sup_{f\in G\setminus\{0\}}R_{\nu}(f)\quad\text{and }\;\lambda^{D}_{n}=\inf_{\substack{G\subset H^{1}_{0}(\Omega)\\
\dim G=n
}
}\ \sup_{f\in G\setminus\{0\}}R_{\nu}(f).
\]
 Equivalently, the infima may be restricted to $H^{N}_{\nu}$ and
$H^{D}_{\nu}$, respectively.
\end{prop}

Define the eigenvalue counting function 
\[
N^{B}(x)\coloneqq\card\{n\in\Np:\lambda^{B}_{n}\leq x\}.
\]
The lower and upper spectral dimensions are 
\[
\underline{s}^{B}\coloneqq\liminf_{x\to\infty}\frac{\log N^{B}(x)}{\log x},\qquad\overline{s}^{B}\coloneqq\limsup_{x\to\infty}\frac{\log N^{B}(x)}{\log x}.
\]
 If they coincide, their common value is denoted by $s^{B}$. If the
Dirichlet and Neumann upper spectral dimensions coincide, we write
$\overline{s}$; if both spectral dimensions exist and coincide, we
write $s$.

By the min--max principle we obtain the following two observations.
\begin{cor}[Dirichlet--Neumann comparison]
\label{cor:Dirichlet-Neumann-comparison} For every $n\in\Np$, $\lambda^{N}_{n}\leq\lambda^{D}_{n}.$
Consequently, $N^{D}(x)\leq N^{N}(x)$, $\underline{s}^{D}\leq\underline{s}^{N}$,
and $\overline{s}^{D}\leq\overline{s}^{N}$.
\end{cor}

\begin{proof}
The inclusion $H^{1}_{0}(\Omega)\subset H^{1}(\Omega)$ shows that
every Dirichlet trial subspace is admissible for the Neumann min--max
problem. Hence $\lambda^{N}_{n}\leq\lambda^{D}_{n}$ for every $n\in\Np$.
The assertions for the counting functions and spectral dimensions
follow immediately.
\end{proof}

\begin{cor}[Orthogonal trial functions]
\label{cor:orthogonal-trial-functions}Let $\mathcal{F}\subset H^{1}_{0}(\Omega)\setminus\{0\}$
be finite. Assume that distinct members are orthogonal both in $L^{2}(\Omega,\nu)$
and with respect to the gradient energy. Then 
\[
N^{D}\left(\max_{f\in\mathcal{F}}R_{\nu}(f)\right)\geq\card(\mathcal{F}).
\]
 The analogous statement holds in the Neumann case with $\mathcal{F}\subset H^{1}(\Omega)$.
\end{cor}

\begin{proof}
The span of $\mathcal{F}$ has dimension $\card(\mathcal{F})$. By
the two orthogonality assumptions, every non-zero linear combination
has Rayleigh quotient bounded above by the largest quotient of a member
of $\mathcal{F}$. Apply \prettyref{prop:unified-min-max}.
\end{proof}

\subsection{Partition functions}

\label{subsec:partition-functions-intro}

Throughout, we choose $\Omega=\Q\coloneqq(0,1)^{d}$, $d\geq2$, and
let $\nu$ be a finite non-zero Borel measure satisfying $\supp\nu\subset\Q.$
We abbreviate $H^{1}\coloneqq H^{1}(\Q)$, $H^{1}_{0}\coloneqq H^{1}_{0}(\Q)$,
and $\mathcal{C}^{\infty}_{c}\coloneqq\mathcal{C}^{\infty}_{c}(\Q)$.
Let $\D_{n}$ denotes the collection of dyadic half-open cubes of
level $n$ contained in $\overline{\Q}$. We put $\D\coloneqq\bigcup_{n\geq0}\D_{n}.$
For a cube $Q$, write $\ell(Q)$ for its side length.

For $Q\in\D$, let $\D(Q)\coloneqq\{\widetilde{Q}\in\D:\widetilde{Q}\subset Q\}.$
For a non-negative monotone set function $\J:\D\longrightarrow[0,\infty),$
where monotonicity means that $\J(Q')\leq\J(Q)$ whenever $Q'\subset Q$,
define 
\begin{equation}
\tau_{\J,n}(q)\coloneqq\frac{1}{\log2^{n}}\log\sum_{Q\in\D_{n}}\J(Q)^{q},\qquad q\geq0,\label{eq:parptionfunction_tau_n}
\end{equation}
 and the $\J$\emph{-partition function} 
\begin{equation}
\tau_{\J}(q)\coloneqq\limsup_{n\to\infty}\tau_{\J,n}(q).\label{eq:partitionfunction}
\end{equation}
 We use the convention $0^{0}=0$, so cubes satisfying $\J(Q)=0$
do not contribute when $q=0$. Define the critical exponent 
\[
q_{\J}\coloneqq\inf\{q\geq0:\tau_{\J}(q)<0\}.
\]
 Further details on partition functions are postponed to \prettyref{sec:Partition-functions}.

For $a\in\R$, $b>0$, and $Q\in\D$, define 
\begin{equation}
\J_{\nu,a,b}(Q)\coloneqq\begin{cases}
{\displaystyle \sup_{\widetilde{Q}\in\D(Q)}\nu(\widetilde{Q})^{b}\bigl|\log\Lambda(\widetilde{Q})\bigr|,} & a=0,\\
\sup_{\widetilde{Q}\in\D(Q)}\nu(\widetilde{Q})^{b}\Lambda(\widetilde{Q})^{a}, & a\neq0.
\end{cases}\label{eq:SpectralSetFunction}
\end{equation}
The supremum over descendants is essential: it makes $\J_{\nu,a,b}$
monotone under inclusion and records the worst scale below $Q$. Under
the parameter condition $b\dim_{\infty}(\nu)+ad>0$ used below, this
set function is finite and uniformly vanishing at small scales. For
$t\geq2$, write $\J_{\nu,t}\coloneqq\J_{\nu,\,2/d-1,\,2/t},$ and
set $\J_{\nu}\coloneqq\J_{\nu,2}.$

The $L^{q}$\emph{-spectrum} of $\nu$ is the special case of a $\J$-partition
function with $\J\coloneqq\nu|_{\mathcal{D}}$ and we write $\beta_{\nu}(q)\coloneqq\limsup_{n\to\infty}\beta_{\nu,n}(q),$
where 
\[
\beta_{\nu,n}(q)\coloneqq\frac{1}{\log2^{n}}\log\sum_{Q\in\D_{n}}\nu(Q)^{q}.
\]
For later reference, define the \emph{lower partition function} by
$\underline{\tau}_{\J}(q)\coloneqq\liminf_{n\to\infty}\tau_{\J,n}(q).$
Thus $\underline{\tau}_{\J}(q)\leq\tau_{\J}(q).$ Similarly, the \emph{lower
$L^{q}$-spectrum} is given by $\underline{\beta}_{\nu}(q)\coloneqq\liminf_{n\to\infty}\beta_{\nu,n}(q).$

\subsection{The lower \texorpdfstring{$L^{\infty}$}{L-infinity} and Minkowski
dimension}

For a monotone $\J:\D\longrightarrow[0,\infty)$, we define the \emph{lower
$L^{\infty}$-dimension} by 
\[
\dim_{\infty}(\J)\coloneqq\liminf_{n\to\infty}\frac{\log\max_{Q\in\D_{n}}\J(Q)}{\log2^{-n}}.
\]
If $\J$ is given by a Borel measure $\nu|_{\mathcal{D}}$ we write
$\dim_{\infty}(\nu)$; note that in this case one may equivalently
consider the supremum over all balls of radius $r>0$, divide by $\log r$
and take the liminf; see, for example, \cite{MR3236784}.

If $\dim_{\infty}(\nu)<d-2,$ the form construction fails; this is
proved in \prettyref{prop:non-contin_Embedding}. If $\dim_{\infty}(\nu)>d-2,$
that is the \emph{Hu--Lau--Ngai condition} holds, then the extension
$\iota:H^{1}(\Q)\longrightarrow L^{2}(\Q,\nu)$ is compact; its restriction
to $H^{1}_{0}(\Q)$ is compact as well. See \cite{MR2261337} and
the estimates below \prettyref{prop:Sobolev-embedding-continuity}.

Furthermore, we define the upper \emph{Minkowski dimension} of the
support of $\nu$ by
\[
\overline{\dim}_{M}(\supp\nu)\coloneqq\limsup_{n\to\infty}\frac{\log\card\{Q\in\D_{n}:\nu(Q)>0\}}{\log2^{n}}
\]
and use the abbreviation $\overline{\dim}_{M}(\nu)\coloneqq\overline{\dim}_{M}(\supp\nu).$

\subsection{Adaptive partitions}

Let $\Pi_{\J}$ denote the family of finite collections obtained by
retaining the positive $\J$ cubes from a dyadic partition of $\Q$.
Define 
\[
M_{\J}(x)\coloneqq\inf\left\{ \card(P):P\in\Pi_{\J},\ \max_{Q\in P}\J(Q)<x^{-1}\right\} .
\]

The lower and upper partition entropies are 
\[
\underline{h}_{\J}\coloneqq\liminf_{x\to\infty}\frac{\log M_{\J}(x)}{\log x},\qquad\overline{h}_{\J}\coloneqq\limsup_{x\to\infty}\frac{\log M_{\J}(x)}{\log x}.
\]

For $\alpha>0$, define 
\[
\mathcal{N}_{\alpha,\J}(n)\coloneqq\card B_{\alpha,\J}(n)\,\,\,\,\;\text{with }\quad B_{\alpha,\J}(n)\coloneqq\left\{ Q\in\D_{n}:\J(Q)\geq2^{-\alpha n}\right\} .
\]
Set 
\[
\underline{F}_{\J}(\alpha)\coloneqq\liminf_{n\to\infty}\frac{\log^{+}\mathcal{N}_{\alpha,\J}(n)}{\log2^{n}},\qquad\overline{F}_{\J}(\alpha)\coloneqq\limsup_{n\to\infty}\frac{\log^{+}\mathcal{N}_{\alpha,\J}(n)}{\log2^{n}}.
\]
 and 
\[
\underline{F}_{\J}\coloneqq\sup_{\alpha>0}\frac{\underline{F}_{\J}(\alpha)}{\alpha},\qquad\overline{F}_{\J}\coloneqq\sup_{\alpha>0}\frac{\overline{F}_{\J}(\alpha)}{\alpha}.
\]

\subsection{Statement of the main results}

\label{subsec:Main-results}
\begin{namedthm}[Theorem A]
 \label{thmA:MainUpperBound_General} Let $\J:\D\longrightarrow[0,\infty)$
be monotone and uniformly vanishing. Suppose that for every $Q\in\D$
and every $u\in\mathcal{C}^{\infty}(\overline{Q})$ satisfying $\int_{Q}u\d\Lambda=0,$
we have 
\[
\|u\|^{2}_{L^{2}(Q,\nu)}\leq\J(Q)\|\nabla u\|^{2}_{L^{2}(Q)}.
\]
Then the Neumann form-domain embedding is compact and, for all sufficiently
large $x$, $N^{N}(x)\leq M_{\J}(x).$ Consequently, 
\[
\underline{s}^{N}\leq\underline{h}_{\J},\qquad\overline{s}^{N}\leq\overline{h}_{\J}.
\]
\end{namedthm}
\begin{namedthm}[Theorem B]
 \label{thmB:IntroGeneralPrincipleLowerBound} Let $\J:\D\longrightarrow[0,\infty)$
be monotone and satisfy $\dim_{\infty}(\J)>0$. Assume that $\Delta^{D}_{\nu}$
has compact resolvent. Suppose that, for every sufficiently small
dyadic cube $Q$ with $\J(Q)>0$, there exists a non-negative, non-zero
function $\psi_{Q}\in\mathcal{C}^{\infty}_{c}\left(\R^{d}\right)$
satisfying $\supp\psi_{Q}\subset\langle\mathring{Q}\rangle_{3}$ and
\[
\|\psi_{Q}\|^{2}_{L^{2}(\nu)}\geq\J(Q)\|\nabla\psi_{Q}\|^{2}_{L^{2}(\R^{d})}.
\]
 Then 
\[
\underline{F}_{\J}\leq\underline{s}^{D},\qquad\overline{F}_{\J}\leq\overline{s}^{D}.
\]
\end{namedthm}
Use the abbreviations $\mathfrak{q}\coloneqq q_{\J_{\nu}}$, $\underline{F}\coloneqq\underline{F}_{\J_{\nu}}$,
$\overline{F}\coloneqq\overline{F}_{\J_{\nu}}$, $\overline{h}\coloneqq\overline{h}_{\J_{\nu}}$,and
$\underline{h}\coloneqq\lim_{t\downarrow2}\underline{h}_{\J_{\nu,t}}$
(the limit exists, as explained in \prettyref{sec:technical-prerequisites}).
\begin{namedthm}[Theorem C]
 \label{thmC:MainChain} Assume, in addition to the standing hypotheses
on $\nu$, that $\dim_{\infty}(\nu)>d-2$. Then 
\[
\underline{F}\leq\underline{s}^{D}\leq\underline{s}^{N}\leq\underline{h}\leq\overline{h}=\overline{s}=\mathfrak{q}=\overline{F}.
\]
\end{namedthm}
\begin{figure}
\center{\begin{tikzpicture}[scale=0.7, every node/.style={transform shape},line cap=round,line join=round,>=triangle 45,x=1cm,y=1cm] 
\begin{axis}[ x=2.7cm,y=2.7cm, axis lines=middle, axis line style={very thick},ymajorgrids=false, xmajorgrids=false, grid style={thick,densely dotted,black!20}, xlabel= {$q$}, ylabel= {$\tau_{\J_\nu}(q)$}, xmin=-0.49 , xmax=3.5 , ymin=-0.3, ymax=2.5,x tick style={color=black}, xtick={0,1,2,2.48,3},xticklabels = {0,1,2,$\mathfrak{q}$,3},  ytick={0,1,1.47,2},yticklabels = {0,1,$\dim_\infty(\nu)$,2}] \clip(-0.5,-0.3) rectangle (4,4); 
\draw[line width=1pt,smooth,samples=180,domain=-0.3:3.4] plot(\x,{log10(0.08^((\x))+0.2^((\x))+0.36^((\x))+0.36^((\x)))/log10(2)+\x}); 
\draw [line width=01pt,dotted, domain=-0.05 :4.4] plot(\x,{(((log10(0.36))/(log10(2))+1)*(\x-1)+1)});
\draw [line width=01pt,dashed, domain=-0.05 :4.4] plot(\x,{(1-\x)+1)});
 
\node[circle,draw] (c) at (2.48 ,0 ){\,};

\draw [line width=.7pt,dotted, gray] (1 ,0.)--(1,1); 
\draw [line width=.7pt,dotted, gray] (0  ,1.0 )-- (1,1);
\draw (0.07 ,2.13 ) node[anchor=north west] {$ \displaystyle{\overline{\dim}_M(\nu)}$}; 
\end{axis} 
\end{tikzpicture}}

\caption{Partition function $\tau_{\J}$ for the self-similar measure $\nu$
supported on the $3$-dimensional\emph{ Sierpi\'{n}ski tetrahedron}
with all four contraction ratios equal $1/2$ and with probability
vector $\left(0.36,0.36,0.2,0.08\right)$. The bounds for $\overline{s}=\mathfrak{q}$
as obtained in \prettyref{cor:upper-spectral-bounds} for this setting
are the zeros of the dashed line $x\protect\mapsto-x\left(\tau_{\J}\left(0\right)-1\right)+\tau_{\J}\left(0\right)$
and the dotted line $x\protect\mapsto\left(1-x\right)\left(\dim_{\infty}\left(\nu\right)-1\right)+1$.}\label{fig:Moment-generating-function}
\end{figure}

\begin{defn}
The measure $\nu$ is called \emph{multifractal regular} if $\underline{F}=\overline{F}.$

It is called \emph{partition-function regular} if at least one of
the following holds:
\begin{enumerate}[label=(\roman*)]
\item there exists $\varepsilon>0$ such that $\tau_{\J_{\nu}}(t)=\underline{\tau}_{\J_{\nu}}(t)$
for every $\mathfrak{q}-\varepsilon<t<\mathfrak{q};$
\item $\tau_{\J_{\nu}}(\mathfrak{q})=\underline{\tau}_{\J_{\nu}}(\mathfrak{q})$
and $\tau_{\J_{\nu}}$ is differentiable at $\mathfrak{q}$.
\end{enumerate}
\end{defn}

\begin{namedthm}[Theorem D]
 \label{thmD:PF-regular} Assume $\dim_{\infty}(\nu)>d-2$. If $\nu$
is partition-function regular, then it is multifractal regular, and
the Dirichlet and Neumann spectral dimensions exist, coincide, and
satisfy $s^{D}=s^{N}=\mathfrak{q}.$ For $d=2$, differentiability
of $\beta_{\nu}$ at $1$ is sufficient, and then $s^{D}=s^{N}=1.$
\end{namedthm}
\begin{cor}
\label{cor:upper-spectral-bounds} Under the hypotheses of \prettyref{thmC:MainChain},
the upper spectral dimension satisfies 
\[
\frac{d}{2}\leq\frac{\overline{\dim}_{M}(\nu)}{\overline{\dim}_{M}(\nu)-d+2}\leq\overline{s}\leq\frac{\dim_{\infty}(\nu)}{\dim_{\infty}(\nu)-d+2}.
\]
 In particular, if $d=2$, then $\overline{s}=1.$
\end{cor}

In \prettyref{fig:Moment-generating-function} the above inequality
is illustrated for a self-similar measure.
\begin{namedthm}[Theorem E]
\label{thmE:Self-Conformal} Let $\nu$ be a self-conformal measure
supported in $\Q$, generated by finitely many conformal contractions,
with possible overlaps. Assume $\dim_{\infty}(\nu)>d-2.$ Then, for
every $q\geq0$, 
\[
\tau_{\J_{\nu}}(q)=\beta_{\nu}(q)+(d-2)q.
\]
 The measure is partition-function regular, and $s^{D}=s^{N}=\mathfrak{q}.$
In particular, if $d=2$, then $s^{D}=s^{N}=1.$
\end{namedthm}

\section*{List of notation}

\addcontentsline{toc}{section}{List of notation}
\begin{description}[style=multiline,leftmargin=1.3cm]
\par 
\item [{$\Q$}] The open unit cube $\Q=(0,1)^{d}.$
\item [{$\D_{n},\D$}] The level-$n$ dyadic cubes and $\D=\bigcup_{n\geq0}\D_{n}.$
\item [{$\ensuremath{\langle Q\rangle_{c}}$}] The cube concentric with
$Q$ whose side length is $c$ times that of $Q$.
\item [{$\Lambda$}] The $d$-dimensional Lebesgue measure.
\item [{$H^{1}(\Omega)$}] The Sobolev space $W^{1,2}(\Omega)$, used for
Neumann boundary conditions.
\item [{$H^{1}_{0}(\Omega)$}] The closure of $\mathcal{C}^{\infty}_{c}(\Omega)$
in $H^{1}(\Omega)$, used for Dirichlet boundary conditions.
\item [{$\iota_{0},\iota$}] The pointwise restriction on $\mathcal{C}^{\infty}(\overline{\Omega})$
and its bounded extension to $H^{1}(\Omega)$, respectively.
\item [{$\iota^{B}$}] The injective restriction of $\iota$ to the reduced
Sobolev space $H^{B}_{\nu}$.
\item [{$\mathcal{N}^{B},H^{B}_{\nu}$}] For $B=N$ (respectively, $B=D$),
the kernel of $\iota$ on $H^{1}(\Omega)$ (respectively, $H^{1}_{0}(\Omega)$)
and its $\langle\cdot,\cdot\rangle_{\iota}$-orthogonal complement
in that space.
\item [{$\EN,\ED$}] The Neumann and Dirichlet gradient-energy forms on
$L^{2}(\nu)$.
\item [{$\DeltaN,\DeltaD$}] The associated non-negative Neumann and Dirichlet
Kreĭn--Feller operators.
\item [{$N^{N},N^{D}$}] The Neumann and Dirichlet eigenvalue counting
functions.
\item [{$\underline{s}^{B}$,~$\overline{s}^{B}$}] The lower and upper
spectral dimensions ($B\in\{D,N\}$).
\item [{$s^{B}$}] The spectral dimension when the lower and upper spectral
dimensions coincide ($B\in\{D,N\}$).
\item [{$\dim_{\infty}(\nu)$}] The lower $L^{\infty}$-dimension of the
measure.
\item [{$\overline{\dim}_{M}(\nu)$}] The upper Minkowski dimension of
$\supp\nu$.
\item [{$\J_{\nu,a,b}$}] The dyadic set function ${\displaystyle \sup_{\widetilde{Q}\subset Q}\nu(\widetilde{Q})^{b}\Lambda(\widetilde{Q})^{a}}$
$\left(a\neq0\right)$, ${\displaystyle \sup_{\widetilde{Q}\subset Q}\nu(\widetilde{Q})^{b}|\log\Lambda(\widetilde{Q})|}$
$\left(a=0\right)$.
\item [{$\J_{\nu,t}$}] The local measure--Sobolev set function $\J_{\nu,t}=\J_{\nu,\,2/d-1,\,2/t}$.
\item [{$\J_{\nu}$}] The spectral set function $\J_{\nu}=\J_{\nu,\,2}.$
\item [{$\tau_{\J,n},\tau_{\J}$}] The finite-scale and upper partition
functions.
\item [{$\underline{\tau}_{\J}$}] The lower partition function.
\item [{$q_{\J}$}] The critical exponent $q_{\J}=\inf\{q\geq0:\tau_{\J}(q)<0\}.$
\item [{$\mathfrak{q}$}] The spectral critical exponent $\mathfrak{q}=q_{\J_{\nu}}.$
\item [{$\beta_{\nu,n},\beta_{\nu}$}] The finite-scale and upper $L^{q}$-spectra
of $\nu$.
\item [{$\underline{\beta}_{\nu}$}] The lower $L^{q}$-spectrum.
\item [{$M_{\J}(x)$}] The optimal cardinality of a dyadic partition with
maximal $\J$-mass less than $1/x$.
\item [{$\ensuremath{\underline{h}_{\J},\overline{h}_{\J}}$}] The lower
and upper partition entropies.
\item [{$\ensuremath{\underline{F}_{\J},\overline{F}_{\J}}$}] The lower
and upper optimised coarse multifractal dimensions.
\end{description}

\section{Partition functions and $L^{q}$-spectra}

\label{sec:Partition-functions}

We now study the partition functions introduced in \prettyref{subsec:partition-functions-intro}.
Throughout this section, $\nu$ is a finite non-zero Borel measure
supported in $\Q$.

\subsection{Basic properties of the partition function}

\label{subsec:partition-function-properties}

Let $\J:\D\longrightarrow[0,\infty)$ be a non-trivial monotone set
function. Recall the definition of the corresponding partition function
$\tau_{\J}$ given in \prettyref{eq:parptionfunction_tau_n} and \prettyref{eq:partitionfunction}
and its critical exponent $q_{\J}$.

Hölder’s inequality, followed by the subadditivity of $\limsup$,
gives the following.
\begin{lem}
\label{lem:partition-function-convexity} For every $n\in\N$ for
which $\max_{Q\in\D_{n}}\J(Q)>0$, the function $q\longmapsto\tau_{\J,n}(q)$
is convex on $[0,\infty)$. Consequently, $\tau_{\J}$ is convex wherever
it is finite.

\label{lem:uniform-vanishing} If $\dim_{\infty}(\J)>0,$ then $\J$
is \emph{uniformly vanishing}, that is, ${\displaystyle \lim_{n\to\infty}\max_{Q\in\D_{n}}\J(Q)=0}$.
\end{lem}

\begin{proof}
For $0<a<\dim_{\infty}(\J)$ and all sufficiently large $n$, $\max_{Q\in\D_{n}}\J(Q)\leq2^{-an},$
which proves the claim.
\end{proof}

If $\nu$ is a probability measure, then $\beta_{\nu}(1)=0$, $\beta_{\nu}(0)=\overline{\dim}_{M}(\nu)$
and for the asymptotic direction we have $\lim_{q\to\infty}\beta_{\nu}\left(q\right)/q=-\dim_{\infty}\left(\nu\right)$.
Furthermore, the functions $q\longmapsto\beta_{\nu,n}(q)$ are convex
and non-increasing and the same is true of $\beta_{\nu}$ on $\R_{\geq0}$.
For every $Q\in\D$ satisfying $\nu(Q)>0$, restriction cannot decrease
the lower $L^{\infty}$-dimension: $\dim_{\infty}(\nu)\leq\dim_{\infty}(\nu|_{Q}).$
Consequently, the condition $\dim_{\infty}(\nu)>d-2$ passes to every
non-zero dyadic restriction.

\subsection{Comparison between spectral partition functions and $L^{q}$-spectra}

\label{subsec:spectral-Lq-comparison}

For $a\in\R$, $b>0$, and $Q\in\D$, recall the definition of $\J_{\nu,a,b}(Q)$
stated in \prettyref{eq:SpectralSetFunction}.
\begin{prop}
\label{prop:spectral-Lq-comparison} Let $a\in\R$ and $b>0$, and
assume $b\dim_{\infty}(\nu)+ad>0.$ Then the following assertions
hold.
\begin{enumerate}[label=(\roman*)]
\item If $a>0$, then $\tau_{\J_{\nu,a,b}}(q)=\beta_{\nu}(bq)-adq\qquad(q\geq0).$
\item If $a<0$, then 
\[
\beta_{\nu}(bq)-adq\leq\tau_{\J_{\nu,a,b}}(q)\leq\beta_{\nu}\left(q\left(b+\frac{ad}{\dim_{\infty}(\nu)}\right)\right)
\]
 for every $q\geq0$. In particular, $\tau_{\J_{\nu,a,b}}(0)=\beta_{\nu}(0).$
\end{enumerate}
\end{prop}

\begin{proof}
Suppose first that $a>0$. Since both $Q\longmapsto\nu(Q)\quad\text{and}\quad Q\longmapsto\Lambda(Q)$
are monotone under inclusion, the supremum in the definition of $\J_{\nu,a,b}(Q)$
is attained at $Q$. Hence $\J_{\nu,a,b}(Q)=\nu(Q)^{b}\Lambda(Q)^{a}.$
For $Q\in\D_{n}$, $\Lambda(Q)^{a}=2^{-adn},$ and therefore 
\[
\sum_{Q\in\D_{n}}\J_{\nu,a,b}(Q)^{q}=2^{-adqn}\sum_{Q\in\D_{n}}\nu(Q)^{bq}.
\]
 Taking logarithms, dividing by $\log2^{n}$, and passing to the upper
limit gives $\tau_{\J_{\nu,a,b}}(q)=\beta_{\nu}(bq)-adq.$

Now suppose $a<0$. The lower bound follows from $\J_{\nu,a,b}(Q)\geq\nu(Q)^{b}\Lambda(Q)^{a}.$
Thus $\tau_{\J_{\nu,a,b}}(q)\geq\beta_{\nu}(bq)-adq.$

For the upper bound, choose $-ad/b<s<\dim_{\infty}(\nu).$ For all
sufficiently large $n$ and every $Q\in\D_{n}$, $\nu(Q)\leq2^{-sn}.$
Hence $n\leq-\log_{2}\nu(Q)/s$ whenever $\nu(Q)>0$. Therefore, 
\[
\nu(Q)^{bq}\Lambda(Q)^{aq}=\nu(Q)^{bq}2^{-adqn}\leq\nu(Q)^{bq}2^{(adq/s)\log_{2}\nu(Q)}=\nu(Q)^{q(b+ad/s)}.
\]
The same estimate applies to all descendants in the supremum defining
$\J_{\nu,a,b}$. Consequently, $\tau_{\J_{\nu,a,b}}(q)\leq\beta_{\nu}\left(q\left(b+ad/s\right)\right).$
Letting $s\nearrow\dim_{\infty}(\nu)$ and using continuity of the
convex function $\beta_{\nu}$ on the interior of its finite domain
gives the claim.
\end{proof}

\begin{cor}
\label{cor:linear-Lq-spectrum} Assume $a\neq0$, $b\dim_{\infty}(\nu)+ad>0,$
and suppose that $\beta_{\nu}(q)=D(1-q)$, $q\geq0,$ for some $D>0$.
Then $\tau_{\J_{\nu,a,b}}(q)=D-q(bD+ad).$ In particular, $q_{\J_{\nu,a,b}}=D/\left(bD+ad\right).$
\end{cor}

\begin{proof}
If $a>0$, this follows directly from \prettyref{prop:spectral-Lq-comparison}.
If $a<0$, the lower and upper bounds from that proposition agree
because $\beta_{\nu}$ is affine.
\end{proof}

\subsection{The logarithmic case}

\label{subsec:logarithmic-partition-function}

The logarithmic case $a=0$ is relevant in dimension $d=2$.
\begin{prop}
\label{prop:logarithmic-Lq} Assume $\dim_{\infty}(\nu)>0.$ Then,
for every $b>0$ and $q\geq0$, $\tau_{\J_{\nu,0,b}}(q)=\beta_{\nu}(bq).$
If $\beta_{\nu,n}(bq)$ converges, then $\lim_{n\to\infty}\tau_{\J_{\nu,0,b},n}(q)=\beta_{\nu}(bq).$
\end{prop}

\begin{proof}
Fix $0<\eta<\dim_{\infty}(\nu).$ For all sufficiently large $n$
and all $Q\in\D_{n}$, $\nu(Q)\leq2^{-\eta n}.$ Thus, for every descendant
$\widetilde{Q}\subset Q$, $|\log\Lambda(\widetilde{Q})|\leq C_{\eta}|\log\nu(\widetilde{Q})|.$
For every $\delta\in(0,b)$, sufficiently small positive masses satisfy
$|\log t|^{q}\leq t^{-q\delta}.$ Consequently, $\J_{\nu,0,b}(Q)^{q}\leq C\sup_{\widetilde{Q}\subset Q}\nu(\widetilde{Q})^{q(b-\delta)}.$
Monotonicity of $\nu$ gives $\J_{\nu,0,b}(Q)^{q}\leq C\nu(Q)^{q(b-\delta)}.$
On the other hand, $\J_{\nu,0,b}(Q)^{q}\geq|\log\Lambda(Q)|^{q}\nu(Q)^{bq}.$
The logarithmic factor is subexponential in the dyadic level, and
therefore $\beta_{\nu}(bq)\leq\tau_{\J_{\nu,0,b}}(q)\leq\beta_{\nu}(q(b-\delta)).$
Letting $\delta\searrow0$ proves the first assertion. The same estimates
with lower and upper limits prove the second assertion when the $L^{q}$-spectrum
exists as a limit.
\end{proof}

\begin{cor}
\label{cor:d-two-critical-exponent} If $d=2$ and $\dim_{\infty}(\nu)>0,$
then $\tau_{\J_{\nu}}(1)=\beta_{\nu}(1)=0$ and $\mathfrak{q}=1.$
\end{cor}

\begin{proof}
For $d=2$, $\J_{\nu}=\J_{\nu,0,1}.$ Apply \prettyref{prop:logarithmic-Lq}
with $b=q=1$.
\end{proof}

\subsection{General bounds for the critical exponent}

\label{subsec:critical-exponent-bounds}
\begin{lem}
\label{lem:critical-exponent-bounds} Assume $b\dim_{\infty}(\nu)+ad>0.$
Then the following assertions hold.
\begin{enumerate}[label=(\roman*)]
\item If $a\leq0$, then 
\[
\frac{\overline{\dim}_{M}(\nu)}{b\overline{\dim}_{M}(\nu)+ad}\leq q_{\J_{\nu,a,b}}\leq\frac{\dim_{\infty}(\nu)}{b\dim_{\infty}(\nu)+ad}.
\]
\item If $a>0$, then 
\[
\frac{\dim_{\infty}(\nu)}{b\dim_{\infty}(\nu)+ad}\leq q_{\J_{\nu,a,b}}\leq\frac{\overline{\dim}_{M}(\nu)}{b\overline{\dim}_{M}(\nu)+ad}.
\]
\item We always have $\dim_{\infty}(\nu)\leq\overline{\dim}_{M}(\nu).$
\item For all $c>0$ we have $\tau_{\J_{c\nu,a,b}}=\tau_{\J_{\nu,a,b}}$.
\end{enumerate}
\end{lem}

\begin{proof}
If $a=0$, then $q_{\J_{\nu,0,b}}=1/b$ by \prettyref{prop:logarithmic-Lq},
so the two bounds coincide. For $a\neq0$, the inequalities follow
from \prettyref{prop:spectral-Lq-comparison} and the standard bounds
$\dim_{\infty}(\nu)(1-q)\leq\beta_{\nu}(q)\leq\overline{\dim}_{M}(\nu)(1-q)$
with the direction chosen according to whether $q\leq1$ or $q\geq1$.
The remaining assertions are immediate from the definitions.
\end{proof}

For the spectral set function $\J_{\nu}=\J_{\nu,\,2/d-1,\,1},$ we
have $a=2/d-1$ and hence $ad=2-d.$ Therefore, 
\[
\frac{d}{2}\leq\frac{\overline{\dim}_{M}(\nu)}{\overline{\dim}_{M}(\nu)-d+2}\leq\mathfrak{q}\leq\frac{\dim_{\infty}(\nu)}{\dim_{\infty}(\nu)-d+2}.
\]

\section{Partition functions for important classes of measures}

\label{sec:special-partition-functions}

\subsection{Absolutely continuous measures}

\label{subsec:Absolutely-continuous-measures}
\begin{prop}
\label{prop:absolute-continuous-partition} Let $\nu$ be a non-zero
measure absolutely continuous with respect to $\Lambda$, with density
$f\in L^{r}(\Q)$ for some $r\geq d/2.$ Then, for every $q\in[0,r]$,
$\beta_{\nu}(q)=\lim_{n\to\infty}\beta_{\nu,n}(q)=d(1-q).$ If $d>2$,
then 
\[
\tau_{\J_{\nu}}(q)=\lim_{n\to\infty}\tau_{\J_{\nu},n}(q)=d-2q
\]
 for every $q\in[0,r]$.
\end{prop}

\begin{proof}
After normalising, assume that $\nu$ is a probability measure. Convexity
and the identities $\beta_{\nu}(0)\leq d,\qquad\beta_{\nu}(1)=0$
give one side of the estimate.

For $q\in[0,1]$, Jensen's inequality gives $\nu(Q)^{q}\geq\left(\int_{Q}f^{q}\d\Lambda\right)\Lambda(Q)^{q-1}.$
For $q\geq1$, it gives $\nu(Q)^{q}\leq\left(\int_{Q}f^{q}\d\Lambda\right)\Lambda(Q)^{q-1}.$
Summing over $Q\in\D_{n}$ yields $\beta_{\nu}(q)=d(1-q)$ for $q\in[0,r]$.

Now assume $d>2$. The lower bound $\tau_{\J_{\nu}}(q)\geq\beta_{\nu}(q)+(d-2)q=d-2q$
follows because $\J_{\nu}(Q)\geq\nu(Q)\Lambda(Q)^{2/d-1}.$

For $d/2\leq q\leq r,$ Jensen's inequality gives 
\[
\nu(Q)^{q}\Lambda(Q)^{2q/d-q}\leq\left(\int_{Q}f^{q}\d\Lambda\right)\Lambda(Q)^{2q/d-1}.
\]
 Since $2q/d-1\geq0,$ the right-hand side is monotone under inclusion.
Hence 
\[
\sum_{\widetilde{Q}\in\D_{n}}\J_{\nu}(\widetilde{Q})^{q}\leq2^{-n(2q-d)}\|f\|^{q}_{L^{q}(\Q)}.
\]
 Thus $\tau_{\J_{\nu}}(q)\leq d-2q$ for $q\in[d/2,r]$. Convexity
and the values at $0$ and $d/2$ extend the equality to $q\in[0,d/2]$.
All of these estimates hold at finite level. Together with the already
proved convergence of $\beta_{\nu,n}$, they also show that $\tau_{\J_{\nu},n}(q)$
converges to $d-2q$ throughout the stated range.
\end{proof}

\subsection{Ahlfors--David regular measures}

\label{subsec:Ahlfors-David}

A finite Borel measure $\nu$ is called $\alpha$-Ahlfors--David
regular if there exists $c\geq1$ such that $c^{-1}r^{\alpha}\leq\nu(B(x,r))\leq cr^{\alpha}$
for every $x\in\supp\nu$ and $0<r\leq\diam(\supp\nu).$

For such a measure, $\beta_{\nu}(q)=\alpha(1-q).$ Moreover, $\dim_{\infty}(\nu)=\overline{\dim}_{M}(\nu)=\alpha.$
\begin{prop}
\label{prop:Ahlfors-partition-function} Let $d>2$, and let $\nu$
be a non-trivial $\alpha$-Ahlfors--David regular measure with $d-2<\alpha\leq d.$
Then, for every $q\geq0$, 
\[
\tau_{\J_{\nu}}(q)=\lim_{n\to\infty}\tau_{\J_{\nu},n}(q)=\alpha-(\alpha-d+2)q.
\]
 Consequently, $\nu$ is partition-function regular and $\mathfrak{q}=\frac{\alpha}{\alpha-d+2}.$
\end{prop}

\begin{proof}
Set $\gamma\coloneqq\alpha-d+2>0$ and $\D_{n}(\nu)\coloneqq\{Q\in\D_{n}:\nu(Q)>0\}.$
The Ahlfors bounds and standard dyadic packing estimates give constants
$c,C>0$ such that, for all sufficiently large $n,$
\[
c2^{\alpha n}\leq\card\D_{n}(\nu)\leq C2^{\alpha n},\qquad\nu(Q)\leq C2^{-\alpha n}\quad(Q\in\D_{n}(\nu)).
\]
 Put $M\coloneqq\nu(\Q)$ and, after enlarging $C$ if necessary,
set 
\[
\mathcal{G}_{n}\coloneqq\left\{ Q\in\D_{n}(\nu):\nu(Q)\geq\frac{M}{2C}2^{-\alpha n}\right\} .
\]
 The cubes outside $\mathcal{G}_{n}$ carry total mass at most $M/2,$
whereas every cube in $\mathcal{G}_{n}$ has mass at most $C2^{-\alpha n}.$
Since $\sum_{Q\in\D_{n}(\nu)}\nu(Q)=M,$ it follows that, after changing
$c,$ 
\[
\card\mathcal{G}_{n}\geq c2^{\alpha n},\qquad\nu(Q)\geq c2^{-\alpha n}\quad(Q\in\mathcal{G}_{n}).
\]
 If $Q\in\D_{n}(\nu)$ and $\widetilde{Q}\in\D_{m}(Q)$ with $m\geq n,$
then the upper Ahlfors bound yields 
\[
\nu(\widetilde{Q})\Lambda(\widetilde{Q})^{2/d-1}\leq C2^{-\alpha m}2^{(d-2)m}=C2^{-\gamma m}\leq C2^{-\gamma n}.
\]
 Taking the supremum over all dyadic descendants gives $\J_{\nu}(Q)\leq C2^{-\gamma n}.$
Conversely, for $Q\in\mathcal{G}_{n},$ using $Q$ itself in the defining
supremum gives 
\[
\J_{\nu}(Q)\geq\nu(Q)\Lambda(Q)^{2/d-1}\geq c2^{-\alpha n}2^{(d-2)n}=c2^{-\gamma n}.
\]

Consequently, for every fixed $q\geq0,$ with the convention $0^{0}=0,$
there are constants $c_{q},C_{q}>0$ such that 
\[
c_{q}2^{n(\alpha-\gamma q)}\leq\sum_{Q\in\D_{n}}\J_{\nu}(Q)^{q}\leq C_{q}2^{n(\alpha-\gamma q)}.
\]
 After taking base-$2$ logarithms and dividing by $n,$ we obtain
that the partition function exists as a limit for every $q\geq0$.
Consequently, $\nu$ is partition-function regular and its unique
zero is $\mathfrak{q}=\alpha/(\alpha-d+2).$
\end{proof}

\subsection{Self-conformal measures}

\label{subsec:Conformal-IFS}

Let $\overline{\Q}=[0,1]^{d}.$ Let $U\subset\R^{d}$ be an open neighbourhood
of $\overline{\Q}$. A $C^{1}$-map $S:U\longrightarrow U$ is conformal
if, for every $x\in U$, $|S'(x)y|=\|S'(x)\||y|\qquad(y\in\R^{d}),$
and $\|S'(x)\|>0.$

Let $\{S_{i}\}^{l}_{i=1}$ be finitely many injective conformal contractions
satisfying $S_{i}(\overline{\Q})\subset\overline{\Q}.$ Assume that
the maps do not all have the same fixed point. There is a unique non-empty
compact set $K\subset\overline{\Q}$ such that $K=\bigcup^{l}_{i=1}S_{i}(K).$
This and the measure construction below are the classical fixed-point
results of \cite{MR625600}.

Let $(p_{i})^{l}_{i=1}$ be a positive probability vector. The associated
self-conformal measure is the unique probability measure satisfying
$\nu=\sum^{l}_{i=1}p_{i}\,\nu\circ S^{-1}_{i}.$ For compatibility
with the standing assumptions, we additionally assume $K\subset\Q.$

The conformal system satisfies bounded distortion: there exists $C\geq1$
such that, for every finite word $u$, $C^{-1}\leq\frac{\|S_{u}'(x)\|}{\|S_{u}'(y)\|}\leq C\qquad(x,y\in U).$
See, for example, \cite[Theorem 4.1.3]{MR2003772}.

The $L^{q}$-spectrum of a self-conformal measure exists as a limit
for $q>0$; see \cite[Theorem 1.1]{MR1838304}.
\begin{prop}
\label{prop:self-conformal-partition-function} Let $d>2$, and let
$\nu$ be a self-conformal measure satisfying $K\subset\Q$ and $\dim_{\infty}(\nu)>d-2$.
Then, for every $q>0$, 
\[
\tau_{\J_{\nu}}(q)=\underline{\tau}_{\J_{\nu}}(q)=\beta_{\nu}(q)+(d-2)q.
\]
\end{prop}

\begin{proof}
Set $a\coloneqq2-d<0.$ The assumption $a>-\dim_{\infty}(\nu)$ implies
$\sup_{Q\in\D}\nu(Q)\Lambda(Q)^{a/d}<\infty.$ For $n\in\N$, let
$W_{n}$ be the stopping family of words $u$ satisfying $\diam(S_{u}(\Q))\leq2^{-n}<\diam(S_{u^{-}}(\Q)).$
This is a finite partition of the symbolic space. Bounded distortion
implies that inverse images of dyadic descendants of a level-$n$
cube can be covered by a uniformly bounded number of dyadic cubes
at the corresponding relative scale.

Using the self-conformal identity and the bounded overlap of neighbouring
dyadic cubes, one obtains a constant $C_{q}\geq1$ such that 
\[
\sum_{Q\in\D_{n}}\J_{\nu}(Q)^{q}\leq C_{q}\,2^{(d-2)qn}\sum_{Q\in\D_{n}}\nu(Q)^{q}.
\]
 The reverse inequality follows from $\J_{\nu}(Q)\geq\nu(Q)\Lambda(Q)^{2/d-1}.$
Consequently, 
\[
\beta_{\nu,n}(q)+(d-2)q\leq\tau_{\J_{\nu},n}(q)\leq\beta_{\nu,n}(q)+(d-2)q+o(1).
\]
 Since $\beta_{\nu,n}(q)$ converges, the assertion follows.
\end{proof}

\section{Adaptive approximation and partition entropy}

\label{sec:OptimalPartitions}

We now recall the abstract approximation results from \cite{KN2023}.
Let $\J:\D\to[0,\infty)$ be monotone and locally non-vanishing, meaning
that each $Q\in\D$ with $\J(Q)>0$ contains a proper dyadic subcube
$Q'\subsetneq Q$ with $\J(Q')>0$. Assume also that $\J(\Q)>0$,
$\dim_{\infty}(\J)>0$, and that there exist $q>0$ and $c\in\R$
such that $\tau_{\J,n}(q)\geq c$ for all sufficiently large $n$.
The first two displayed assumptions make $\J$ non-trivial and uniformly
vanishing, respectively. Local non-vanishing is automatic for every
$\J_{\nu,a,b}$ used below, because a positive-mass dyadic cube has
a positive-mass child.

Recall the partition entropies $\underline{h}_{\J}$ and $\overline{h}_{\J},$
and the optimised coarse multifractal dimensions $\underline{F}_{\J}$
and $\overline{F}_{\J}.$
\begin{thm}[Adaptive approximation principle]
\label{thm:adaptive-approximation} Under the assumptions above,
$\underline{F}_{\J}\leq\underline{h}_{\J}\leq\overline{h}_{\J}=\overline{F}_{\J}=q_{\J}.$
\end{thm}

\begin{proof}
This is \cite[Theorems 1.3 and 1.7]{KN2023}, applied to the dyadic
tree $\D$.
\end{proof}

\begin{prop}[Partition-function regularity]
\label{prop:PF-implies-MF} Suppose that $\J$ is partition-function
regular in the sense that, for $q_{\J}$, one of the following holds:
\begin{enumerate}[label=(\roman*)]
\item there exists $\varepsilon>0$ such that $\tau_{\J}(t)=\underline{\tau}_{\J}(t)\qquad(q_{\J}-\varepsilon<t<q_{\J});$
\item $\tau_{\J}(q_{\J})=\underline{\tau}_{\J}(q_{\J})$ and $\tau_{\J}$
is differentiable at $q_{\J}$.
\end{enumerate}
Then $\underline{F}_{\J}=\overline{F}_{\J}=q_{\J}.$
\end{prop}

\begin{proof}
This is the regularity consequence of \cite[Theorem 1.12]{KN2023}.
\end{proof}

\subsection{The logarithmic partition entropy in dimension two}

For $a\in\R$ and $b>0$, write $M_{a,b}\coloneqq M_{\J_{\nu,a,b}},$
$\underline{h}_{a,b}\coloneqq\underline{h}_{\J_{\nu,a,b}},$ $\overline{h}_{a,b}\coloneqq\overline{h}_{\J_{\nu,a,b}}.$
\begin{prop}
\label{prop:dimension-two-partition-entropy} If $\dim_{\infty}(\nu)>0,$
then, for every $b>0$, $\underline{h}_{0,b}=\overline{h}_{0,b}=q_{\J_{\nu,0,b}}=1/b.$
\end{prop}

\begin{proof}
By \prettyref{prop:logarithmic-Lq}, $\tau_{\J_{\nu,0,b}}(q)=\beta_{\nu}(bq).$
Since $\beta_{\nu}(1)=0,$ the critical exponent is $q_{\J_{\nu,0,b}}=1/b.$
Furthermore, $\nu(Q)^{b}|\log\Lambda(Q)|\geq c\,\nu(Q)^{b}$ on all
sufficiently small dyadic cubes. Hence the lower partition entropy
of $\J_{\nu,0,b}$ dominates that of the set function $Q\mapsto\nu(Q)^{b}$.
The result now follows from \prettyref{thm:adaptive-approximation}
and \cite[Corollary 1.6]{KN2023}.
\end{proof}

\subsection{Scaling in the parameters}
\begin{prop}
\label{prop:partition-entropy-parameters} Assume $b\dim_{\infty}(\nu)+ad>0.$
If $a<0$, then 
\[
\overline{h}_{a,b}=\frac{1}{b}\overline{h}_{a/b,1}\leq\frac{1}{b}q_{\J_{\nu,a/b,1}}\leq\frac{\dim_{\infty}(\nu)}{b\dim_{\infty}(\nu)+ad}.
\]
 If $a>0$, then 
\[
\overline{h}_{a,b}\leq q_{\J_{\nu,a,b}}=\inf\{q>0:\beta_{\nu}(bq)<adq\},
\]
 and 
\[
q_{\J_{\nu,a,b}}\leq\frac{\overline{\dim}_{M}(\nu)}{b\overline{\dim}_{M}(\nu)+ad}.
\]
\end{prop}

\begin{proof}
Since $\J_{\nu,a,b}(Q)=\J_{\nu,a/b,1}(Q)^{b},$ we have $M_{a,b}(x)=M_{a/b,1}(x^{1/b}).$
Therefore, $\overline{h}_{a,b}=\frac{1}{b}\overline{h}_{a/b,1}.$
The remaining estimates follow from \ref{thm:adaptive-approximation} and \ref{lem:critical-exponent-bounds}.
\end{proof}

For $a=\frac{2}{d}-1$ and $t>2$, the relevant set function is $\J_{\nu,t}=\J_{\nu,\,2/d-1,\,2/t}.$
\begin{cor}
\label{cor:t-dependent-partition-entropy} Assume $\dim_{\infty}(\nu)>d-2$
and $d>2.$ For $2<t<\frac{2\dim_{\infty}(\nu)}{d-2},$ we have 
\[
\overline{h}_{\J_{\nu,t}}\leq\frac{t}{2}q_{\J_{\nu,\,(2/d-1)t/2,\,1}}\leq\frac{\dim_{\infty}(\nu)}{2\dim_{\infty}(\nu)/t+2-d}.
\]
 Moreover, $\lim_{t\downarrow2}q_{\J_{\nu,\,(2/d-1)t/2,\,1}}=\mathfrak{q}.$
\end{cor}

\begin{proof}
The first two inequalities follow from \prettyref{prop:partition-entropy-parameters}.

Set $a\coloneqq2/d-1<0$ and $I\coloneqq(0,2\dim_{\infty}(\nu)/(d-2))$,
so that $2\in I$. For fixed $q\geq0$ and $n$, Hölder's inequality
shows that
\[
t\longmapsto\tau_{\J_{\nu,\,at/2,\,1},n}(q)
\]
is convex on $I$. The same is true of its upper limit. The parameter
condition $\dim_{\infty}(\nu)+adt/2>0$ ensures finiteness on $I$,
so this convex function is continuous at $t=2$. Consequently, for
every $q\geq0$,
\[
\tau_{\J_{\nu,\,at/2,\,1}}(q)\longrightarrow\tau_{\J_{\nu,\,a,\,1}}(q)=\tau_{\J_{\nu}}(q)\qquad(t\downarrow2).
\]
The positive lower $L^{\infty}$-dimension of these set functions
also makes each partition function strictly decreasing through its
unique zero. The conclusion now follows from \prettyref{lem:convergence-of-zeros}.
\end{proof}

\begin{lem}[Convergence of zeros]
\label{lem:convergence-of-zeros}Let $f_{n},f\colon[c,d]\to\mathbb{R}$
be decreasing functions such that $f_{n}\to f$ pointwise. If $f_{n}$
and $f$ have unique zeros $x_{n}$ and $x$, respectively, then $x_{n}\to x$.
\end{lem}

\begin{proof}
For $c\leq a<x<b\leq d$, uniqueness of the zero and monotonicity
give$f(a)>0>f(b).$ Hence $f_{n}(a)>0>f_{n}(b)$ for all sufficiently
large $n$, and therefore $a<x_{n}<b$. Letting $a\uparrow x$ and
$b\downarrow x$ proves the claim. At an endpoint, the same argument
is applied from one side.
\end{proof}

\subsection{A coarse multifractal lower estimate}

Let $(n_{k})$ be a subsequence and suppose $B(q)\coloneqq\lim_{k\to\infty}\tau_{\J,n_{k}}(q)$
exists at the chosen $q$. Let $[a',b']=-\partial B(q),$ where $\partial B(q)$
is the convex subdifferential.
\begin{prop}[\cite{KN2023}]
\label{prop:coarse-subdifferential-bound} Assume the hypotheses
above. Then 
\[
a'\geq\dim_{\infty}(\J)\quad\text{and }\quad\frac{a'q+B(q)}{b'}\leq\sup_{\alpha>0}\liminf_{k\to\infty}\frac{\log\mathcal{N}_{\alpha,\J}(n_{k})}{\alpha\log2^{n_{k}}}.
\]
 If, in addition, $B(q)=\tau_{\J}(q),$ then $-\partial B(q)\subset-\partial\tau_{\J}(q).$
Writing $[a,b]\coloneqq-\partial\tau_{\J}(q),$ and assuming $0\leq q\leq q_{\J},$
we have 
\[
\frac{aq+\tau_{\J}(q)}{b}\leq\frac{a'q+B(q)}{b'}.
\]
\end{prop}

This estimate is used below to obtain lower bounds on the lower spectral
dimension.

\section{Technical prerequisites}

\label{sec:technical-prerequisites}

In this section, we collect the Sobolev extension, min--max, and
smoothing results needed for the proofs of the upper and lower spectral
bounds.

Throughout this section, we assume that $\card(\supp\nu)=\infty,$
or equivalently, that $L^{2}(\Q,\nu)$ is infinite-dimensional.

Since partition entropies are invariant under positive rescaling of
$\nu$, we may assume $\nu(\Q)=1$. Then $t\mapsto\J_{\nu,t}\coloneqq\J_{\nu,\,2/d-1,\,2/t}$
is pointwise non-decreasing for $t\geq2$; hence so are $t\mapsto\underline{h}_{\J_{\nu,t}}$
and $t\mapsto\overline{h}_{\J_{\nu,t}}$. Thus both right limits at
$2$ exist in the extended half-line. Recall that $\J_{\nu}=\J_{\nu,2}$.

\subsection{Poincaré--Wirtinger norms and extension operators}

\label{subsec:extension-operators}

The following norm equivalence will be used uniformly on cubes.
\begin{lem}
\label{lem:equivalent-cube-norm} There exists a constant $c_{d}>0$,
depending only on $d$, such that, for every axis-parallel cube $Q$
with $\ell(Q)\leq1$ and every $u\in H^{1}(Q)$, 
\[
c_{d}\|u\|^{2}_{H^{1}(Q)}\leq\|\nabla u\|^{2}_{L^{2}(Q)}+\frac{1}{\Lambda(Q)}\left|\int_{Q}u\d\Lambda\right|^{2}\leq C_{d}\|u\|^{2}_{H^{1}(Q)},
\]
 where $C_{d}>0$ also depends only on $d$.
\end{lem}

\begin{proof}
The upper bound follows from the Cauchy--Schwarz inequality: 
\[
\frac{1}{\Lambda(Q)}\left|\int_{Q}u\d\Lambda\right|^{2}\leq\|u\|^{2}_{L^{2}(Q)}.
\]
 For the lower bound, write $u=u_{Q}+(u-u_{Q}),$ $u_{Q}\coloneqq\frac{1}{\Lambda(Q)}\int_{Q}u\d\Lambda.$
The scale-invariant Poincaré--Wirtinger inequality on cubes gives
\[
\|u-u_{Q}\|_{L^{2}(Q)}\leq C_{d}\,\operatorname{diam}(Q)\|\nabla u\|_{L^{2}(Q)}.
\]
 Since $\ell(Q)\leq1$, we have $\operatorname{diam}(Q)\leq\sqrt{d}$;
hence the constant in the lower estimate is uniform in $Q$.
\end{proof}

Let $\mathfrak{E}_{\Q}:H^{1}(\Q)\longrightarrow H^{1}(\R^{d})$ be
a bounded Stein extension operator. Such an operator exists for every
Lipschitz domain; see \cite{MR0290095}. If $Q$ is a cube and $T_{Q}:\Q\longrightarrow Q$
is the affine similarity mapping $\Q$ onto $Q$, define $\mathfrak{E}_{Q}u\coloneqq\bigl(\mathfrak{E}_{\Q}(u\circ T_{Q})\bigr)\circ T^{-1}_{Q}.$
\begin{lem}
\label{lem:uniform-cube-extension} There exists $C_{\mathrm{ext}}>0$,
depending only on $d$ and on the chosen extension operator for $\Q$,
such that 
\[
\|\mathfrak{E}_{Q}u\|_{H^{1}(\R^{d})}\leq C_{\mathrm{ext}}\left(\|\nabla u\|^{2}_{L^{2}(Q)}+\frac{1}{\ell(Q)^{2}\Lambda(Q)}\left|\int_{Q}u\d\Lambda\right|^{2}\right)^{1/2},
\]
 for every axis-parallel cube $Q$ with $\ell(Q)\leq1$ and every
$u\in H^{1}(Q)$.
\end{lem}

\begin{proof}
Put $r\coloneqq\ell(Q)$ and $w\coloneqq u\circ T_{Q}$. Boundedness
of $\mathfrak{E}_{\Q}$, the affine change of variables, and \prettyref{lem:equivalent-cube-norm}
on the unit cube give
\[
\begin{aligned}\|\mathfrak{E}_{Q}u\|^{2}_{H^{1}(\R^{d})} & \leq C\bigl(r^{d}\|\mathfrak{E}_{\Q}w\|^{2}_{L^{2}(\R^{d})}+r^{d-2}\|\nabla\mathfrak{E}_{\Q}w\|^{2}_{L^{2}(\R^{d})}\bigr)\\
 & \leq C\left((1+r^{2})\|\nabla u\|^{2}_{L^{2}(Q)}+\frac{1+r^{-2}}{\Lambda(Q)}\left|\int_{Q}u\d\Lambda\right|^{2}\right)\\
 & \leq2C\left(\|\nabla u\|^{2}_{L^{2}(Q)}+\frac{1}{r^{2}\Lambda(Q)}\left|\int_{Q}u\d\Lambda\right|^{2}\right).
\end{aligned}
\]
In the last step we used $r\leq1$. The asserted estimate follows.
\end{proof}

\subsection{Local Sobolev embeddings}

\label{subsec:local-Sobolev-embeddings}

The measure-theoretic Sobolev embedding theorem of Maz'ya, Preobrazhenskii,
and Adams gives the following estimate; see \cite{MR336316,MR743823,MR817985,MR2777530}.
\begin{prop}[Local measure-Sobolev estimate]
\label{prop:local-measure-Sobolev} Let $d\geq2$ and $t>2$. There
exists a constant $C=C(d,t)>0$ such that, for every finite Borel
measure $\nu$, every dyadic cube $Q$, and every $v\in\mathcal{C}^{\infty}_{c}(\R^{d})$,
\[
\|v\|^{2}_{L^{t}(\R^{d},\nu|_{Q})}\leq C\J_{\nu,t}(Q)\|v\|^{2}_{H^{1}(\R^{d})}.
\]
\end{prop}

Combining Hölder's inequality with \prettyref{prop:local-measure-Sobolev}
and the extension operator gives the local $L^{2}(\nu)$-estimate
used in \prettyref{thmA:MainUpperBound_General} .
\begin{cor}
\label{cor:local-L2-estimate} Let $d\geq2$ and $t>2$. There exists
$C=C(d,t)>0$ such that, for every dyadic cube $Q\in\D$ and every
$u\in\mathcal{C}^{\infty}(\overline{Q})$ satisfying $\int_{Q}u\d\Lambda=0,$
we have 
\[
\|u\|^{2}_{L^{2}(Q,\nu)}\leq C\nu(Q)^{1-2/t}\J_{\nu,t}(Q)\|\nabla u\|^{2}_{L^{2}(Q)}\leq C\nu(\Q)^{1-2/t}\J_{\nu,t}(Q)\|\nabla u\|^{2}_{L^{2}(Q)}.
\]
\end{cor}

\begin{proof}
Hölder's inequality gives $\|u\|_{L^{2}(Q,\nu)}\leq\nu(Q)^{1/2-1/t}\|u\|_{L^{t}(Q,\nu)}.$
If $\J_{\nu,t}(Q)=+\infty$, the assertion is trivial. Otherwise,
by density, the estimate in \prettyref{prop:local-measure-Sobolev}
extends from $\mathcal{C}^{\infty}_{c}(\R^{d})$ to $H^{1}(\R^{d})$.
Apply \prettyref{prop:local-measure-Sobolev} to $\mathfrak{E}_{Q}u$.
Since $\int_{Q}u\d\Lambda=0,$ \prettyref{lem:uniform-cube-extension}
gives $\|\mathfrak{E}_{Q}u\|_{H^{1}(\R^{d})}\leq C_{\mathrm{ext}}\|\nabla u\|_{L^{2}(Q)}.$
Squaring and combining these estimates gives 
\[
\|u\|^{2}_{L^{2}(Q,\nu)}\leq C\nu(Q)^{1-2/t}\J_{\nu,t}(Q)\|\nabla u\|^{2}_{L^{2}(Q)}.
\]
 The second inequality in the statement follows from $\nu(Q)\leq\nu(\Q).$
\end{proof}

\subsection{Continuity and compactness of the Sobolev embedding operator}

\label{subsec:Sobolev-embedding-continuity}
\begin{prop}
\label{prop:Sobolev-embedding-continuity} If $\dim_{\infty}(\nu)>d-2,$
then $\iota_{0}$ extends uniquely to a compact operator
\[
\iota:H^{1}(\Q)\longrightarrow L^{2}(\Q,\nu).
\]
 Its restriction to $H^{1}_{0}(\Q)$ is compact as well.
\end{prop}

\begin{proof}
Choose $t>2$ sufficiently close to $2$ that $(2/t)\dim_{\infty}(\nu)+2-d>0$.
Then $\J_{\nu,t}(Q)\longrightarrow0$ uniformly as the side length
of $Q$ tends to zero. Applying \prettyref{cor:local-L2-estimate}
on a dyadic partition of $\Q$ gives the global continuity estimate.
The uniform vanishing of the local constants yields compactness by
the usual finite-rank approximation argument; see \cite{MR2261337,MR817985,MR2777530,elib_6573}.
\end{proof}

\section{Smooth cutoff functions and failure below the critical dimension}

\label{sec:smoothing}

\subsection{Smooth cutoff functions}

Let $Q$ be a cube of side length $mr$, where $m>1$, and let $Q'\subset Q$
be the concentric cube of side length $r$. Standard mollification
gives a function $\varphi_{Q,m}\in\mathcal{C}^{\infty}_{c}(\R^{d})$
such that $0\leq\varphi_{Q,m}\leq1,$ $\supp\varphi_{Q,m}\subset\mathring{Q},$
$\varphi_{Q,m}=1\quad\text{on }Q',$ and $|\partial_{j}\varphi_{Q,m}|\leq\frac{C_{d}}{r(m-1)}\qquad(j=1,\ldots,d).$
\begin{lem}
\label{lem:mollifier-function} There exists $C=C(d,m)>0$ such that
\[
\frac{{\displaystyle \int_{\R^{d}}|\nabla\varphi_{Q,m}|^{2}\d\Lambda}}{{\displaystyle \int_{\R^{d}}|\varphi_{Q,m}|^{2}\d\nu}}\leq C\frac{\Lambda(Q')^{1-2/d}}{\nu(Q')}
\]
 whenever $\nu(Q')>0$.
\end{lem}

\begin{proof}
The gradient bound gives $\int_{\R^{d}}|\nabla\varphi_{Q,m}|^{2}\d\Lambda\leq C_{d,m}r^{d-2}.$
Since $\varphi_{Q,m}=1$ on $Q'$, $\int_{\R^{d}}|\varphi_{Q,m}|^{2}\d\nu\geq\nu(Q').$
Finally, $r^{d-2}=\Lambda(Q')^{1-2/d}.$
\end{proof}

\subsection{Failure of continuity below dimension $d-2$}
\begin{prop}
\label{prop:non-contin_Embedding} If $\dim_{\infty}(\nu)<d-2,$ then
the restriction map $\iota_{0}:\bigl(\mathcal{C}^{\infty}(\overline{\Q}),\|\cdot\|_{H^{1}(\Q)}\bigr)\longrightarrow L^{2}(\Q,\nu)$
is not continuous.
\end{prop}

\begin{proof}
Choose $\dim_{\infty}(\nu)<a<d-2.$ There exists a sequence of dyadic
cubes $Q_{n}$, with side lengths tending to zero, such that $\nu(Q_{n})\geq\Lambda(Q_{n})^{a/d}.$
Since $\supp\nu\Subset\Q,$ we may assume, after discarding finitely
many terms, that $\langle Q_{n}\rangle_{2}\subset\Q.$

Set $u_{n}\coloneqq\Lambda(\langle Q_{n}\rangle_{2})^{1/d-1/2}\varphi_{\langle Q_{n}\rangle_{2},2}.$
The gradient estimate and the support bound imply $\sup_{n}\|u_{n}\|_{H^{1}(\Q)}<\infty.$
On the other hand, 
\[
\|u_{n}\|^{2}_{L^{2}(\nu)}\geq\Lambda(\langle Q_{n}\rangle_{2})^{2/d-1}\nu(Q_{n})\geq C\Lambda(Q_{n})^{(a+2-d)/d}.
\]
Since $a+2-d<0$ and $\Lambda(Q_{n})\longrightarrow0,$ the right-hand
side tends to $+\infty$. Hence the embedding is not continuous.
\end{proof}

\section{Upper spectral bounds}

\label{sec:Upper-bounds}
\begin{proof}[Proof of \prettyref{thmA:MainUpperBound_General}]
Fix a dyadic partition $P\in\Pi_{\J}.$ The local embedding estimate
implies that every cube with $\J(Q)=0$ is $\nu$-null; this follows
by applying it to finitely many mean-zero smooth functions having
no common zero on $Q$. Hence the cubes in $P$ cover $\Q$ up to
a $\nu$-null set. We first verify the compactness asserted in the
theorem. Let $P_{n}$ be the positive-$\J$ cubes in $\D_{n}$ and
set
\[
A_{n}f\coloneqq\sum_{Q\in P_{n}}\left(\frac{1}{\Lambda(Q)}\int_{Q}f\d\Lambda\right)\mathbbm{1}_{Q}.
\]
For smooth $f$, the local estimate applied to $f-f_{Q}$ on every
$Q\in P_{n}$ gives
\[
\|\iota f-A_{n}f\|^{2}_{L^{2}(\nu)}\leq\max_{Q\in\D_{n}}\J(Q)\|\nabla f\|^{2}_{L^{2}(\Q)}.
\]
Since the estimate is uniform on the dense subspace of smooth functions,
it extends to all of $H^{1}(\Q)$. Thus $\iota$ is the operator-norm
limit of the finite-rank operators $A_{n}$ and is compact. Define
\[
\mathcal{F}_{P}\coloneqq\left\{ u\in H^{1}(\Q):\int_{Q}u\d\Lambda=0\text{ for every }Q\in P\right\} .
\]
 This is a closed subspace of $H^{1}(\Q)$. The quotient has dimension
\[
\dim(H^{1}(\Q)/\mathcal{F}_{P})=\card(P),
\]
 because the independent linear constraints are the cube averages.

We first extend the local estimate from smooth functions to all members
of $\mathcal{F}_{P}$. Let $u\in\mathcal{F}_{P}$, and choose $u_{n}\in\mathcal{C}^{\infty}(\overline{\Q})$
such that $u_{n}\longrightarrow u\quad\text{in }H^{1}(\Q).$ For every
$Q\in P$, $\int_{Q}u_{n}\d\Lambda\longrightarrow0.$ Choose functions
$\eta_{Q}\in\mathcal{C}^{\infty}_{c}(Q)$ satisfying $\int_{Q}\eta_{Q}\d\Lambda=1.$
Set 
\[
u_{n}'\coloneqq u_{n}-\sum_{Q\in P}\left(\int_{Q}u_{n}\d\Lambda\right)\eta_{Q}.
\]
 Then $u_{n}'\in\mathcal{C}^{\infty}(\overline{\Q})\cap\mathcal{F}_{P}$
and $u_{n}'\longrightarrow u\quad\text{in }H^{1}(\Q).$

For $u_{n}'$, the local Poincaré estimates give 
\[
\|\iota u_{n}'\|^{2}_{L^{2}(\nu)}=\sum_{Q\in P}\|u_{n}'\|^{2}_{L^{2}(Q,\nu)}\leq\sum_{Q\in P}\J(Q)\|\nabla u_{n}'\|^{2}_{L^{2}(Q)}\leq\max_{Q\in P}\J(Q)\|\nabla u_{n}'\|^{2}_{L^{2}(\Q)}.
\]
Passing to the limit yields $\|\iota u\|^{2}_{L^{2}(\nu)}\leq\max_{Q\in P}\J(Q)\|\nabla u\|^{2}_{L^{2}(\Q)}.$
Thus every non-zero $u\in\mathcal{F}_{P}$ satisfies $R_{\nu}(u)\geq1/\max_{Q\in P}\J(Q).$
Suppose $\max_{Q\in P}\J(Q)<1/x$ and set $m\coloneqq\operatorname{card}(P)$.
Then $R_{\nu}(u)\geq1/\max_{Q\in P}\mathfrak{J}(Q)>x$, $u\in\mathcal{F}_{P}\setminus\{0\}$.
Moreover, $\operatorname{codim}\mathcal{F}_{P}\leq m$. Suppose $N^{N}(x)\geq m+1$,
then the span of the first $m+1$ Neumann eigenfunctions intersects
$\mathcal{F}_{P}$ non-trivially. Every non-zero element $u$ of this
intersection satisfies $R_{\nu}(u)\leq x$, which is a contradiction.
Consequently, $N^{N}(x)\leq m=\card(P).$ Taking the infimum over
all admissible partitions gives $N^{N}(x)\leq M_{\J}(x)$ and consequently
$\underline{s}^{N}\leq\underline{h}_{\J}$ and $\overline{s}^{N}\leq\overline{h}_{\J}.$
\end{proof}

Next we apply \prettyref{thmA:MainUpperBound_General} to a positive
multiple of $\J_{\nu,t}.$
\begin{cor}
\label{cor:UpperBoundSpectralDim} Let $\dim_{\infty}(\nu)>d-2.$
Then 
\[
\underline{s}^{D}\leq\underline{s}^{N}\leq\underline{h},\qquad\overline{s}^{D}\leq\overline{s}^{N}\leq\lim_{t\downarrow2}\overline{h}_{\J_{\nu,t}}\leq\mathfrak{q}.
\]
 If $d=2$, then $\overline{s}^{N}\leq1.$
\end{cor}

\begin{proof}
If $d>2$, choose $2<t<2\dim_{\infty}(\nu)/(d-2)$; if $d=2$, fix
any $t>2$. In either case $\J_{\nu,t}$ is non-negative, monotone,
and uniformly vanishing. By \prettyref{cor:local-L2-estimate}, \prettyref{thmA:MainUpperBound_General}
applies to a positive multiple of $\J_{\nu,t}$. Since the lower and
upper partition entropies are invariant under positive rescaling,
\prettyref{thmA:MainUpperBound_General} gives $\underline{s}^{N}\leq\underline{h}_{\J_{\nu,t}}$
and $\overline{s}^{N}\leq\overline{h}_{\J_{\nu,t}}$.Letting $t\downarrow2$
gives the Neumann bounds.

When $d>2$, \prettyref{cor:t-dependent-partition-entropy} implies
$\lim_{t\downarrow2}\overline{h}_{\J_{\nu,t}}\leq\mathfrak{q}$.

When $d=2$, $\J_{\nu,t}=\J_{\nu,0,2/t}$, so \prettyref{prop:dimension-two-partition-entropy}
gives $\overline{h}_{\J_{\nu,t}}\leq t/2$. Thus $\lim_{t\downarrow2}\overline{h}_{\J_{\nu,t}}\leq1=\mathfrak{q}$,
where the last equality follows from \prettyref{cor:d-two-critical-exponent}.
Finally, all Dirichlet inequalities follow from \prettyref{cor:Dirichlet-Neumann-comparison}.
\end{proof}

\section{Lower spectral bounds}

\label{sec:Lower-bounds}

\subsection{Separated dyadic families}

For a cube $Q$, let $\langle Q\rangle_{c}$ denote the concentric
cube whose side length is $c$ times that of $Q$.
\begin{lem}
\label{lem:separated-dyadic-family} Let $\mathcal{A}\subset\D_{n}$
be finite. Then there exists $\mathcal{E}\subset\mathcal{A}$ such
that $\card(\mathcal{E})\geq5^{-d}\card(\mathcal{A})$ and $\langle\mathring{Q}\rangle_{3}\cap\langle\mathring{Q}'\rangle_{3}=\emptyset$
for all distinct $Q,Q'\in\mathcal{E}$.
\end{lem}

\begin{proof}
Choose cubes greedily. Whenever a cube $Q$ is selected, remove all
level-$n$ cubes meeting $\langle\mathring{Q}\rangle_{5}$. At most
$5^{d}$ cubes are removed at each step. The remaining selected cubes
have disjoint triple enlargements.
\end{proof}

Before stating the proof of our abstract lower bound we prove the
following preparatory lemma.
\begin{lem}
\label{lem:GeneralPrincipleLowerBound} Assume the hypotheses of \prettyref{thmB:IntroGeneralPrincipleLowerBound}.
Fix $\alpha>0$, and set $n_{\alpha,x}\coloneqq\left\lfloor \frac{\log_{2}x}{\alpha}\right\rfloor .$
Then, for all sufficiently large $x$, $5^{-d}\mathcal{N}_{\alpha,\J}(n_{\alpha,x})-1\leq N^{D}(x).$
\end{lem}

\begin{proof}
For $Q\in B_{\alpha,\J}(n)$, the trial-function estimate implies
$\langle\mathring{Q}\rangle_{3}\cap\supp\nu\neq\varnothing$. Since
$\supp\nu\Subset\mathfrak{Q}$, all sufficiently fine such cubes satisfy
$\langle\mathring{Q}\rangle_{3}\Subset\mathfrak{Q}$. Hence the corresponding
functions $\psi_{Q}$ belong to $C^{\infty}_{c}(\mathfrak{Q})$. Apply
\prettyref{lem:separated-dyadic-family} to obtain $\mathcal{E}_{n}\subset B_{\alpha,\J}(n)$
with $\card(\mathcal{E}_{n})\geq5^{-d}\mathcal{N}_{\alpha,\J}(n)-1$
and pairwise disjoint triple enlargements.

For every $Q\in\mathcal{E}_{n}$, choose the test function $\psi_{Q}$
from \prettyref{thmB:IntroGeneralPrincipleLowerBound}. Its Rayleigh
quotient satisfies $R_{\nu}(\psi_{Q})\leq\frac{1}{\J(Q)}\leq2^{\alpha n}.$
The functions have disjoint supports, so they are orthogonal both
in $L^{2}(\nu)$ and with respect to the Dirichlet energy. Therefore,
\prettyref{cor:orthogonal-trial-functions} gives $N^{D}(2^{\alpha n})\geq\card(\mathcal{E}_{n}).$
Taking $n=n_{\alpha,x}$ proves the claim, up to replacing $x$ by
a fixed multiplicative constant, which does not affect the asymptotic
exponents.
\end{proof}

\begin{proof}[Proof of \prettyref{thmB:IntroGeneralPrincipleLowerBound}]
By \prettyref{lem:GeneralPrincipleLowerBound}, $5^{-d}\mathcal{N}_{\alpha,\J}(n_{\alpha,x})-1\leq N^{D}(x).$
Taking lower logarithmic limits gives $\underline{s}^{D}\geq\frac{\underline{F}_{\J}(\alpha)}{\alpha}.$
Taking the supremum over $\alpha>0$ yields $\underline{F}_{\J}\leq\underline{s}^{D}.$

Taking the subsequence $x_{n}\coloneqq2^{\alpha n}$ and upper logarithmic
limits gives $\overline{s}^{D}\geq\frac{\overline{F}_{\J}(\alpha)}{\alpha}.$
Again taking the supremum over $\alpha>0$ yields $\overline{F}_{\J}\leq\overline{s}^{D}.$
\end{proof}

\subsection{Lower test functions for the spectral set function}

For the lower estimate, define $\underline{\J}_{\nu}(Q)\coloneqq\sup_{\widetilde{Q}\in\D(Q)}\nu(\widetilde{Q})\Lambda(\widetilde{Q})^{2/d-1}.$
If $d>2$, this agrees with $\J_{\nu}$. If $d=2$, then $\underline{\J}_{\nu}(Q)=\nu(Q).$
\begin{lem}
\label{lem:dimension-two-J-equivalence} If $d=2$ and $\dim_{\infty}(\nu)>0,$
then $\underline{F}_{\underline{\J}_{\nu}}=\underline{F}_{\J_{\nu}}$
and $\overline{F}_{\underline{\J}_{\nu}}=\overline{F}_{\J_{\nu}}.$
\end{lem}

\begin{proof}
Let $B_{\alpha,\nu}(n)\coloneqq\{Q\in\D_{n}:\nu(Q)\ge2^{-\alpha n}\}.$
Since $d=2$, every $Q\in\D_{n}$ satisfies $|\log\Lambda(Q)|=2n\log2\ge1.$
Consequently, $\J_{\nu}(Q)\ge\nu(Q)|\log\Lambda(Q)|\ge\nu(Q),$ and
hence
\[
B_{\alpha,\nu}(n)\subset B_{\alpha,\J_{\nu}}(n)\qquad(n\ge1).
\]
 Conversely, fix $\delta>1$. Since $\dim_{\infty}(\nu)>0$, uniform
smallness of all dyadic masses implies, on every sufficiently fine
cube and all its descendants,
\[
\nu(R)|\log\Lambda(R)|\leq\nu(R)^{1/\delta}\leq\nu(Q)^{1/\delta},\qquad R\subset Q.
\]
 Hence $\J_{\nu}(Q)\leq\nu(Q)^{1/\delta}$ and therefore
\[
B_{\alpha,\J_{\nu}}(n)\subset B_{\alpha\delta,\nu}(n).
\]
 The two inclusions compare the corresponding lower and upper coarse
counting exponents. Dividing by the scale parameter, then letting
$\delta\downarrow1$, proves both asserted equalities.
\end{proof}

\begin{prop}
\label{prop:LowerBoundFunctions} Assume $\dim_{\infty}(\nu)>d-2$.
There exists $C>0$ such that, for every $Q\in\D$ satisfying $\underline{\J}_{\nu}(Q)>0,$
there is a non-negative, non-zero function $\psi_{Q}\in\mathcal{C}^{\infty}_{c}(\R^{d})$
with $\supp\psi_{Q}\subset\langle\mathring{Q}\rangle_{3}$ and 
\[
\|\psi_{Q}\|^{2}_{L^{2}(\nu)}\geq C\underline{\J}_{\nu}(Q)\|\nabla\psi_{Q}\|^{2}_{L^{2}(\R^{d})}.
\]
\end{prop}

\begin{proof}
Choose $Q_{0}\in\D(Q)$ such that, up to a fixed factor, $\underline{\J}_{\nu}(Q)\leq2\nu(Q_{0})\Lambda(Q_{0})^{2/d-1}.$
Set $\psi_{Q}\coloneqq\varphi_{\langle Q_{0}\rangle_{3},3}.$ Then
$\psi_{Q}=1\quad\text{on }Q_{0},$ and $\supp\psi_{Q}\subset\langle\mathring{Q}_{0}\rangle_{3}\subset\langle\mathring{Q}\rangle_{3}.$
By \prettyref{lem:mollifier-function}, 
\[
\frac{\|\nabla\psi_{Q}\|^{2}_{L^{2}(\R^{d})}}{\|\psi_{Q}\|^{2}_{L^{2}(\nu)}}\leq C\frac{\Lambda(Q_{0})^{1-2/d}}{\nu(Q_{0})}\leq C'\underline{\J}_{\nu}(Q)^{-1}.
\]
\end{proof}

\begin{cor}
\label{cor:lower-spectral-bound-Jnu} If $\dim_{\infty}(\nu)>d-2$,
then $\underline{F}_{\J_{\nu}}\leq\underline{s}^{D}$ and $\overline{F}_{\J_{\nu}}\leq\overline{s}^{D}.$
\end{cor}

\begin{proof}
The hypothesis and \prettyref{prop:Sobolev-embedding-continuity}
imply that $\Delta^{D}_{\nu}$ has compact resolvent; moreover, $\dim_{\infty}(\underline{\J}_{\nu})\geq\dim_{\infty}(\nu)-d+2>0$.
By \prettyref{prop:LowerBoundFunctions}, the hypothesis of \prettyref{thmB:IntroGeneralPrincipleLowerBound}
holds for the positive multiple $C\underline{\J}_{\nu}$. The optimised
coarse multifractal dimensions are invariant under multiplication
of the set function by a positive constant. The assertion follows
by applying \prettyref{thmB:IntroGeneralPrincipleLowerBound} and
then \prettyref{lem:dimension-two-J-equivalence} when $d=2$.
\end{proof}

\section{Completion of the proofs of our main theorems}

\label{sec:MainProofs}

We now combine the upper and lower estimates obtained in the preceding
sections.

\subsection{Proof of the chain of inequalities}
\begin{proof}[Proof of \prettyref{thmC:MainChain}]
The lower spectral estimate \prettyref{cor:lower-spectral-bound-Jnu},
the Dirichlet--Neumann comparison \prettyref{cor:Dirichlet-Neumann-comparison},
and the upper estimates in \prettyref{cor:UpperBoundSpectralDim}
give 
\[
\underline{F}\leq\underline{s}^{D}\leq\underline{s}^{N}\leq\underline{h}\leq\lim_{t\downarrow2}\overline{h}_{\J_{\nu,t}},\qquad\overline{F}\leq\overline{s}^{D}\leq\overline{s}^{N}\leq\lim_{t\downarrow2}\overline{h}_{\J_{\nu,t}}\leq\mathfrak{q}.
\]

The inequality $\underline{h}\leq\lim_{t\downarrow2}\overline{h}_{\J_{\nu,t}}$
follows from $\underline{h}_{\J_{\nu,t}}\leq\overline{h}_{\J_{\nu,t}}$.
Using montonicity and the upper bound, $\overline{h}_{\mathfrak{J}_{\nu,2}}\leq\lim_{t\downarrow2}\overline{h}_{\mathfrak{J}_{\nu,t}}\leq\mathfrak{q}$.
The adaptive approximation theorem \prettyref{thm:adaptive-approximation}
gives $\overline{F}=\overline{h}=\overline{h}_{\mathfrak{J}_{\nu,2}}=\mathfrak{q}$,
closing both chains.
\end{proof}

\subsection{Consequences of partition-function regularity}
\begin{proof}[Proof of \prettyref{thmD:PF-regular}]
Suppose that $\nu$ is partition-function regular. By \prettyref{prop:PF-implies-MF},
$\underline{F}=\overline{F}=\mathfrak{q}.$ The chain of inequalities
in \prettyref{thmC:MainChain} then yields $\mathfrak{q}=\underline{F}\leq\underline{s}^{D}\leq\underline{s}^{N}\leq\overline{s}^{D}=\overline{s}^{N}=\mathfrak{q}.$
Consequently, $\underline{s}^{D}=\overline{s}^{D}=\underline{s}^{N}=\overline{s}^{N}=\mathfrak{q}.$
Thus both spectral dimensions exist and $s^{D}=s^{N}=\mathfrak{q}.$

Suppose that $d=2$ and that $\beta_{\nu}$ is differentiable at $1$.
Multiplying $\nu$ by a positive constant changes the finite-scale
logarithmic moments only by $O(1/n)$, so we may assume that $\nu(\Q)=1$.
Then $\beta_{\nu,n}(1)=0$ for every $n$. By \prettyref{prop:logarithmic-Lq},
$\tau_{\J_{\nu}}=\beta_{\nu}$ and $\tau_{\J_{\nu},n}(1)\to0=\tau_{\J_{\nu}}(1)$.
Thus condition (ii) of partition-function regularity holds at $\mathfrak{q}=1$,
where the value of the critical exponent follows from \prettyref{cor:d-two-critical-exponent}.
The first part of the theorem now gives $s^{D}=s^{N}=1$.
\end{proof}

\subsection{Bounds in terms of geometric dimensions}
\begin{proof}[Proof of \prettyref{cor:upper-spectral-bounds}]
By \prettyref{thmC:MainChain}, $\overline{s}=\mathfrak{q}.$ The
bounds for $\mathfrak{q}$ obtained in \ref{lem:critical-exponent-bounds}
give 
\[
\frac{\overline{\dim}_{M}(\nu)}{\overline{\dim}_{M}(\nu)-d+2}\leq\overline{s}\leq\frac{\dim_{\infty}(\nu)}{\dim_{\infty}(\nu)-d+2}.
\]

Since $\overline{\dim}_{M}(\nu)\leq d,$ and the function $t\longmapsto\frac{t}{t-d+2}$
is decreasing on $(d-2,\infty)$, we obtain $\frac{d}{2}\leq\frac{\overline{\dim}_{M}(\nu)}{\overline{\dim}_{M}(\nu)-d+2}.$
This proves the displayed chain.

If $d=2$, \prettyref{cor:d-two-critical-exponent} gives $\mathfrak{q}=1,$
and hence $\overline{s}=1.$
\end{proof}

\begin{rem}
\label{rem:CanOneHearFractalDrum} For $d>2$, choosing measures for
which $\overline{\dim}_{M}(\nu)$ is close to $d-2$ produces arbitrarily
large upper spectral dimensions.

The estimates can also be inverted. If $d>2$ and $\overline{s}>1$,
then $\overline{\dim}_{M}(\nu)\geq\frac{\overline{s}(d-2)}{\overline{s}-1}\geq\dim_{\infty}(\nu).$
Thus the spectral dimension yields non-trivial information about fractal-geometric
dimensions of the measure.

In dimension $2$, however, $\overline{s}=1$ for every measure satisfying
the standing assumptions. In this case, the first-order spectral exponent
does not distinguish the fractal-geometric nature of the measure.
\end{rem}

\subsection{A lower bound from subdifferentials}

For a convex function $f$, let 
\[
\partial f(q)\coloneqq\left\{ a\in\R:f(t)\geq f(q)+a(t-q)\text{ for all }t\right\} 
\]
 denote its subdifferential.
\begin{prop}
\label{prop:lowerbound-by-subdifferential} Assume $\dim_{\infty}(\nu)>d-2.$
Let $0\leq t\leq\mathfrak{q}$ and suppose $\tau_{\J_{\nu}}(t)=\underline{\tau}_{\J_{\nu}}(t).$
Write $[a,b]\coloneqq-\partial\tau_{\J_{\nu}}(t).$ Then $\frac{at+\tau_{\J_{\nu}}(t)}{b}\leq\underline{s}^{D}.$
\end{prop}

\begin{proof}
Choose a subsequence along which the finite-scale partition functions
converge at $t$. Apply \prettyref{prop:coarse-subdifferential-bound}
to obtain the corresponding lower bound for the optimised lower coarse
multifractal dimension. Then apply $\underline{F}_{\J_{\nu}}\leq\underline{s}^{D}$
from \prettyref{cor:lower-spectral-bound-Jnu}.
\end{proof}

\section{Applications and spectral asymptotics}

\label{sec:applications}

\subsection{Absolutely continuous measures}
\begin{prop}
\label{prop:spec-absolutely-continuous} Let $\nu$ be absolutely
continuous with respect to $\Lambda$, with density $f\in L^{r}(\Q)$
for some $r>\frac{d}{2}.$ Then the Dirichlet and Neumann spectral
dimensions exist and satisfy $s^{D}=s^{N}=d/2.$
\end{prop}

\begin{proof}
By \prettyref{prop:absolute-continuous-partition}, $\tau_{\J_{\nu}}(q)=d-2q$
on an open interval containing $d/2$. Hence $\nu$ is partition-function
regular, and $\mathfrak{q}=\frac{d}{2}.$ Apply \prettyref{thmD:PF-regular}.
\end{proof}

The precise Weyl asymptotic in this setting was obtained by Birman
and Solomyak; see \cite{MR0278126,MR0482138}.

\subsection{Critical absolutely continuous examples}

\label{subsec:critical-examples}

The condition $\dim_{\infty}(\nu)>d-2$ is sufficient for compactness.
At the critical value $\dim_{\infty}(\nu)=d-2,$ different behaviours
are possible.
\begin{example}
\label{exa:critical} Let $d=3$ and set $\Q\coloneqq\left(-\frac{1}{4},\frac{3}{4}\right)^{3}.$
For $i=1,2,3$, define 
\[
f_{i}(x,y,z)\coloneqq z^{-2}\bigl(\log(1/z)\bigr)^{k_{i}}\1_{\{0<x,y\leq z<1/2\}},\qquad k_{1}=0,\,k_{2}=-\frac{3}{4},\,k_{3}=1.
\]
Let $\d\nu_{i}\coloneqq f_{i}\d\Lambda.$ Then $\dim_{\infty}(\nu_{i})=1=d-2$
for $i=1,2,3.$ For $Q_{l}\coloneqq[0,2^{-l}]^{3}$, 
\[
\nu_{i}(Q_{l})=\int^{2^{-l}}_{0}\bigl(\log(1/z)\bigr)^{k_{i}}\,\d z,\qquad i=1,2,3.
\]

Consequently, for $i=1,2,3$, ${\displaystyle \dim_{\infty}\left(\nu_{i}\right)=\liminf_{l\to\infty}\frac{\log\nu_{i}(Q_{l})}{\log2^{-l}}=1.}$

The three measures exhibit distinct form-theoretic behaviours:
\begin{enumerate}[label=(\roman*)]
\item the embeddings associated with $\nu_{1}$ are continuous but not
compact;
\item the embeddings associated with $\nu_{2}$ are compact;
\item the embeddings associated with $\nu_{3}$ are not continuous.
\end{enumerate}
In the second case, the spectral dimension exists and equals $s^{D}=s^{N}=\frac{3}{2}.$
\end{example}

\begin{proof}
The embedding assertions follow from the critical capacity criteria
\cite{MR817985,MR743823,MR2777530}; the dimension and spectral calculations
are given in \cite{elib_6573}.
\end{proof}

\subsection{Spectral asymptotics for Ahlfors--David regular measures}
\begin{cor}
\label{prop:Ahlfors-David-Regular} Under the hypotheses of \prettyref{prop:Ahlfors-partition-function},
\[
N^{D}(x)\asymp N^{N}(x)\asymp x^{\alpha/(\alpha-d+2)}.
\]
\end{cor}

\begin{proof}
Set
\[
\gamma\coloneqq\alpha-d+2>0,\qquad t_{*}\coloneqq\frac{2\alpha}{d-2}>2.
\]
 For each sufficiently large $n,$ let $\D_{n}(\nu)$ and $\mathcal{G}_{n}$
be the families constructed in the proof of \prettyref{prop:Ahlfors-partition-function}.

For the lower bound, apply \prettyref{lem:separated-dyadic-family}
to obtain $\mathcal{E}_{n}\subset\mathcal{G}_{n}$ with $\card\mathcal{E}_{n}\geq5^{-d}\card\mathcal{G}_{n}$
and pairwise disjoint triple enlargements. Since $\supp\nu\Subset\Q,$
the functions $\varphi_{\langle Q\rangle_{3},3}$, $Q\in\mathcal{E}_{n},$
belong to $\mathcal{C}^{\infty}_{c}(\Q)$ for all sufficiently large
$n.$ By \prettyref{lem:mollifier-function}, their Dirichlet Rayleigh
quotients satisfy 
\[
R_{\nu}\bigl(\varphi_{\langle Q\rangle_{3},3}\bigr)\leq C\frac{\Lambda(Q)^{1-2/d}}{\nu(Q)}\leq C2^{n(\alpha-d+2)}=C2^{\gamma n}.
\]
 Their supports are disjoint, so the min--max principle yields 
\[
N^{D}\bigl(C2^{\gamma n}\bigr)\geq\card\mathcal{E}_{n}\geq c2^{\alpha n}.
\]

For the upper bound, let $\widetilde{Q}$ be a sufficiently small
dyadic cube with positive mass. The upper Ahlfors bound and the identity
$2\alpha/t_{*}=d-2$ give 
\[
\nu(\widetilde{Q})^{2/t_{*}}\Lambda(\widetilde{Q})^{2/d-1}\leq C\ell(\widetilde{Q})^{2\alpha/t_{*}+2-d}=C.
\]
 Consequently, $\J_{\nu,t_{*}}(Q)\leq C$ for every sufficiently small
dyadic cube $Q.$ The estimate in \prettyref{cor:local-L2-estimate}
therefore implies that, for $Q\in\D_{n}(\nu)$ and $\int_{Q}u\d\Lambda=0,$
\[
\|u\|^{2}_{L^{2}(Q,\nu)}\leq C\nu(Q)^{1-2/t_{*}}\J_{\nu,t_{*}}(Q)\|\nabla u\|^{2}_{L^{2}(Q)}\leq C2^{-\gamma n}\|\nabla u\|^{2}_{L^{2}(Q)}.
\]

Put $r=2^{-n}$ and choose a maximal $r$-separated family $\{x_{j}\}^{m}_{j=1}\subset\supp\nu$.
The lower Ahlfors bound applied to the disjoint balls $B(x_{j},r/3)$
gives $m\leq Cr^{-\alpha}$. Maximality implies that $B(x_{j},r)$
cover the support, and each such ball meets at most $C_{d}$ level-$n$
dyadic cubes. Therefore, after adjusting $C$, $\card\mathcal{D}_{n}(\nu)\leq C2^{\alpha n}$.
Now set
\[
\mathcal{F}_{n}\coloneqq\left\{ u\in H^{1}(\mathfrak{Q}):\int_{Q}u\,\mathrm{d}\Lambda=0\ \text{for every }Q\in\mathcal{D}_{n}(\nu)\right\} .
\]
Because $\mathcal{D}_{n}(\nu)$ covers $\mathfrak{Q}$ up to a $\nu$-null
set,
\[
\|\iota u\|^{2}_{L^{2}(\nu)}=\sum_{Q\in\mathcal{D}_{n}(\nu)}\|u\|^{2}_{L^{2}(Q,\nu)}\leq C2^{-\gamma n}\sum_{Q\in\mathcal{D}_{n}(\nu)}\|\nabla u\|^{2}_{L^{2}(Q)}\leq C2^{-\gamma n}\|\nabla u\|^{2}_{L^{2}(\mathfrak{Q})}.
\]
Moreover, $\operatorname{codim}\mathcal{F}_{n}\leq\card\mathcal{D}_{n}(\nu)\leq C2^{\alpha n}.$
The min--max principle therefore yields
\[
N^{N}\!\left(c2^{\gamma n}\right)\leq C2^{\alpha n}.
\]

Finally, $N^{D}(x)\leq N^{N}(x)$ by \prettyref{cor:Dirichlet-Neumann-comparison}.
Since both counting functions are non-decreasing, the two estimates
at dyadic scales imply the claim, that is, for all sufficiently large
$x,$ 
\[
cx^{\alpha/\gamma}\leq N^{D}(x)\leq N^{N}(x)\leq Cx^{\alpha/\gamma}.
\]
\end{proof}

This recovers, in particular, the spectral exponent for isotropic
$\alpha$-sets considered by Triebel (see \cite{MR1484417}) and upper
bounds established in (see \cite{MR4331823,MR4484835}) .

\subsection{Self-conformal measures}
\begin{proof}[Proof of \prettyref{thmE:Self-Conformal}]
Assume first that $d>2$. By \ref{prop:self-conformal-partition-function},
$\tau_{\J_{\nu}}(q)=\beta_{\nu}(q)+(d-2)q$ and the finite-scale partition
function converges for every $q>0$. Hence condition (i) of partition-function
regularity holds.

If $d=2$, the $L^{q}$-spectrum of a self-conformal measure exists
as a limit for every $q>0$. Therefore \prettyref{prop:logarithmic-Lq}
gives $\tau_{\J_{\nu}}(q)=\beta_{\nu}(q)$ together with convergence
of the finite-scale partition functions. Again condition (i) holds,
and $\mathfrak{q}=1$ by \prettyref{cor:d-two-critical-exponent}.

In both cases, \prettyref{thmD:PF-regular} yields $s^{D}=s^{N}=\mathfrak{q}$,
with $\mathfrak{q}=1$ when $d=2$.
\end{proof}

\begin{rem}
If the conformal system satisfies the open set condition, the critical
exponent may be expressed through the symbolic pressure. Let $\psi(\omega)\coloneqq\log p_{\omega_{1}}$
and $\varphi(\omega)\coloneqq\log\|S_{\omega_{1}}'(\pi(\sigma\omega))\|.$
Then $\mathfrak{q}$ is the unique zero of $q\longmapsto P\bigl(q(\psi+(2-d)\varphi)\bigr).$

For affine self-similar systems satisfying the open set condition,
one also has the two-sided spectral estimate $N^{D}(x)\asymp N^{N}(x)\asymp x^{\mathfrak{q}};$
see \cite{MR1298682,MR1338787}.
\end{rem}

\subsection{A symbolic proof of the self-similar upper bound}
\begin{rem}
\label{rem:self-similar_IFS_OSC} Let $\nu$ be a self-similar measure
satisfying the open set condition, with contraction ratios $h_{1},\ldots,h_{m}$
and probability weights $p_{1},\ldots,p_{m}.$ Assume $\max_{i}p_{i}h^{2-d}_{i}<1.$
Let $\delta$ be the unique solution of $\sum^{m}_{i=1}\bigl(p_{i}h^{2-d}_{i}\bigr)^{\delta}=1.$

For $t>0$, define the stopping family 
\[
\widetilde{P}_{t}\coloneqq\left\{ \omega:p_{\omega}h^{2-d}_{\omega}<t\leq p_{\omega^{-}}h^{2-d}_{\omega^{-}}\right\} .
\]
 Then $t^{-\delta}\leq\card(\widetilde{P}_{t})\leq Ct^{-\delta}.$
The local embedding estimate on the cylinder $S_{\omega}(\Q)$ gives
\[
\|u\|^{2}_{L^{2}(S_{\omega}(\Q),\nu)}\leq Cp_{\omega}h^{2-d}_{\omega}\|\nabla u\|^{2}_{L^{2}(S_{\omega}(\Q))}
\]
 for zero-mean functions on that cylinder. Applying the codimension
argument from \prettyref{thmA:MainUpperBound_General} yields $N^{N}(c/t)\leq Ct^{-\delta}.$
Thus $N^{N}(x)\leq Cx^{\delta}.$ This reproduces the upper estimate
of \cite{MR1298682,MR1338787} without renewal theory.
\end{rem}

\subsection{An example without spectral dimension}

\label{subsec:Non-existence-of-the}

We now give a three-dimensional example for which the upper and lower
spectral dimensions differ.

Let $\mu$ be the inhomogeneous Cantor measure from \cite[Example 5.5]{KN2022},
translated and rescaled so that $\supp\mu\subset\left(1/4,3/4\right).$
Its finite-scale $L^{q}$-spectra do not converge. The upper and lower
$L^{q}$-spectra satisfy
\[
\beta_{\mu}(q)=\begin{cases}
{\displaystyle 3/8\cdot(1-q),} & 0\leq q\leq1,\\
{\displaystyle 3/10\cdot(1-q),} & q>1,
\end{cases}\;\text{and \, }\underline{\beta}_{\mu}(q)=\begin{cases}
{\displaystyle 3/10\cdot(1-q),} & 0\leq q\leq1,\\
{\displaystyle 3/8\cdot(1-q),} & q>1.
\end{cases}
\]
 Let $\Lambda^{1}$ denote one-dimensional Lebesgue measure restricted
to $(1/4,3/4)$, and define $\nu\coloneqq\Lambda^{1}\otimes\Lambda^{1}\otimes\mu$
on $\Q\subset\R^{3}$. For a level-$n$ dyadic cube $Q=Q_{1}\times Q_{2}\times Q_{3},$
we have $\nu(Q)=2^{-2n}\mu(Q_{3})$ up to a fixed normalising constant.
Since $\Lambda(Q)^{2/3-1}=2^{n},$ it follows that $\J_{\nu}(Q)=2^{-n}\mu(Q_{3})$
up to fixed multiplicative constants. Consequently, $\tau_{\J_{\nu},n}(q)=2-q+\beta_{\mu,n}(q)+o(1).$

For $q>1$, the upper partition function is therefore $\tau_{\J_{\nu}}(q)=2-q+3/10\cdot(1-q),$
whose zero is $\mathfrak{q}=23/13.$ The lower partition function
is $\underline{\tau}_{\J_{\nu}}(q)=2-q+3/8\cdot(1-q),$ whose zero
is $\underline{q}_{\nu}=19/11.$
\begin{example}
\label{exa:nonexistent-spectral-dimension} For the product measure
$\nu$ above, $\overline{s}^{D}=\overline{s}^{N}=23/13,$ while $\underline{s}^{N}\leq19/11.$
Since $19/11<23/13,$ the Neumann spectral dimension does not exist.
\end{example}

\begin{proof}
The upper spectral dimension equals the zero $\mathfrak{q}=23/13$
by \prettyref{thmC:MainChain}. The lower estimate follows by applying
the finite-scale upper-bound construction along a subsequence realising
the lower partition function. Equivalently, the lower partition entropy
is at most $19/11$, and hence $\underline{s}^{N}\leq19/11.$ See
\cite{elib_6573} for the full finite-scale construction.
\end{proof}

\appendix

\section{The form approach}

This appendix contains the six form-theoretic proofs deferred from
the introductory setup: norm equivalence, energy-minimising property,
density and closedness of the forms, the Markov property, the min--max
characterisation, and the Neumann zero mode.
\begin{proof}[Proof of \prettyref{lem:augmented-norm-equivalence}]
The form is positive definite: if $\|f\|_{\iota}=0$, then $f$ is
constant and $\iota f=0$; since $\nu$ is non-zero, this forces $f=0$.
Continuity of $\iota$ gives $\|f\|_{\iota}\leq C\|f\|_{H^{1}(\Omega)}.$

For the converse, suppose that no constant $C>0$ satisfies $\|f\|_{H^{1}(\Omega)}\leq C\|f\|_{\iota}.$
Then there exists a sequence $(f_{n})$ in $H^{1}(\Omega)$ such that
$\|f_{n}\|_{H^{1}(\Omega)}=1$ and 
\[
\|\nabla f_{n}\|_{L^{2}(\Omega)}+\|\iota f_{n}\|_{L^{2}(\Omega,\nu)}\longrightarrow0.
\]
Set $(f_{n})_{\Omega}\coloneqq\frac{1}{\Lambda(\Omega)}\int_{\Omega}f_{n}\d\Lambda.$
By the Poincaré--Wirtinger inequality, we have $\|f_{n}-(f_{n})_{\Omega}\|_{H^{1}(\Omega)}\longrightarrow0.$
The constants $(f_{n})_{\Omega}$ are bounded. After passing to a
subsequence, there exists $c\in\R$ such that $(f_{n})_{\Omega}\longrightarrow c.$
Hence $f_{n}\longrightarrow c$ in $H^{1}(\Omega)$ and continuity
of $\iota$ gives $\iota f_{n}\longrightarrow\iota c$ in $L^{2}(\Omega,\nu)$.
Since $\iota f_{n}\to0$, we have $\iota c=0.$ As $\nu$ is non-zero,
a constant has zero $L^{2}(\nu)$-norm only if it is zero. Thus $c=0$,
contradicting $\|f_{n}\|_{H^{1}(\Omega)}=1.$
\end{proof}

\begin{proof}[Proof of \prettyref{lem:Minimizer}]
We only consider the Neumann case, the Dirichlet case follows exactly
along the same lines. Set $f_{u}\coloneqq(\iota^{N})^{-1}u$. Let
$f\in H^{1}(\Omega)$ satisfy $\iota f=u$, and write
\[
f=f^{\perp}+f^{0},\qquad f^{\perp}\in H^{N}_{\nu},\quad f^{0}\in\mathcal{N}^{N}.
\]
Since $\iota f^{0}=0$, we have $\iota^{N}f^{\perp}=u$. The injectivity
of $\iota^{N}$ therefore gives $f^{\perp}=f_{u}$.

Moreover, $f_{u}\perp_{\iota}f^{0}$ and $\ensuremath{\iota f^{0}=0}$,
hence $0=\langle f_{u},f^{0}\rangle_{\iota}=\int_{\Omega}\nabla f_{u}\cdot\nabla f^{0}\,\mathrm{d}\Lambda,$
and consequently,
\[
\int_{\Omega}|\nabla f|^{2}\,\mathrm{d}\Lambda=\int_{\Omega}|\nabla f_{u}|^{2}\,\mathrm{d}\Lambda+\int_{\Omega}|\nabla f^{0}|^{2}\,\mathrm{d}\Lambda\geq\int_{\Omega}|\nabla f_{u}|^{2}\,\mathrm{d}\Lambda.
\]
This proves the minimising property. If equality holds, then $\nabla f^{0}=0$.
Together with $\iota f^{0}=0$, this gives $\|f^{0}\|_{\iota}=0$,
and therefore $f^{0}=0$. Thus the minimiser is unique.
\end{proof}

\begin{proof}[Proof of \prettyref{prop:forms-density-closedness}]
For $u\in\dom(\mathcal{E}^{B})$, let $f\coloneqq(\iota^{B})^{-1}u.$
Then 
\[
\mathcal{E}^{B}(u,u)+\|u\|^{2}_{L^{2}(\Omega,\nu)}=\|\nabla f\|^{2}_{L^{2}(\Omega)}+\|\iota f\|^{2}_{L^{2}(\Omega,\nu)}=\|f\|^{2}_{\iota}.
\]
Symmetry and non-negativity are immediate from the definition. Hence
$\iota^{B}:\bigl(H^{B}_{\nu},\|\cdot\|_{\iota}\bigr)\longrightarrow\bigl(\dom(\mathcal{E}^{B}),\|\cdot\|_{\mathcal{E}^{B},1}\bigr)$
is an isometric isomorphism. Since $H^{B}_{\nu}$ is closed in $H^{1}(\Omega)$
for $B=N$ and in $H^{1}_{0}(\Omega)$ for $B=D$, it is complete
in the equivalent norm $\|\cdot\|_{\iota}$. Hence the form is closed.

Since $\supp\nu$ is compactly contained in $\Omega$, restrictions
of functions in $\mathcal{C}^{\infty}_{c}(\Omega)$ are dense in $L^{2}(\Omega,\nu)$.
Since $\mathcal{C}^{\infty}_{c}(\Omega)\subset H^{1}_{0}(\Omega)\subset H^{1}(\Omega)$,
both form domains are dense in $L^{2}(\Omega,\nu)$.
\end{proof}

\begin{proof}[Proof of \prettyref{prop:forms-Markovian}]
Let $B\in\left\{ D,N\right\} .$ $u\in\dom(\mathcal{E}^{B})$, set
$f\coloneqq(\iota^{B})^{-1}u$, and let $\eta(t)\coloneqq(0\vee t)\wedge1$.
The Sobolev chain rule and approximation by smooth functions give
$\eta\circ f\in H^{1}(\Omega)$ (and $\eta\circ f\in H^{1}_{0}(\Omega)$
when $B=D$) together with $\iota(\eta\circ f)=\eta\circ u$. Hence
$\eta\circ u\in\dom(\mathcal{E}^{B})$, and the energy-minimising
property in \prettyref{lem:Minimizer} gives
\[
\mathcal{E}^{B}(\eta\circ u,\eta\circ u)\leq\int_{\Omega}|\nabla(\eta\circ f)|^{2}\,\mathrm{d}\Lambda\leq\int_{\Omega}|\nabla f|^{2}\,\mathrm{d}\Lambda=\mathcal{E}^{B}(u,u).
\]
\end{proof}

\begin{proof}[Proof of \prettyref{prop:unified-min-max}]
In the Neumann case, the corresponding Sobolev space is $H^{1}(\Omega)$;
in the Dirichlet case, it is $H^{1}_{0}(\Omega)$. The min--max principle
for the closed form $\mathcal{E}^{B}$, $B\in\left\{ D,N\right\} $
gives 
\begin{equation}
\lambda^{B}_{n}=\inf_{\substack{L\subset\dom(\mathcal{E}^{B})\\
\dim L=n
}
}\sup_{u\in L\setminus\{0\}}\frac{\mathcal{E}^{B}(u,u)}{\|u\|^{2}_{L^{2}(\Omega,\nu)}}.\label{eq:min-max_Pf}
\end{equation}
Let $L\subset\dom(\mathcal{E}^{B})$ be $n$-dimensional. Since $\iota^{B}\colon H^{B}_{\nu}\longrightarrow\dom(\mathcal{E}^{B})$
is a linear isomorphism, the space $F\coloneqq(\iota^{B})^{-1}(L)$
is an $n$-dimensional subspace of $H^{B}_{\nu}$, and hence of the
corresponding Sobolev space. For every $f\in F\setminus\{0\}$, $R_{\nu}(f)=\mathcal{E}^{B}(\iota^{B}f,\iota^{B}f)/\|\iota^{B}f\|^{2}_{L^{2}(\Omega,\nu)}.$
It follows that \prettyref{eq:min-max_Pf} is exactly the min--max
formula restricted to $H^{B}_{\nu}$. In particular, the full Sobolev-space
min--max value is at most $\lambda^{B}_{n}$.

For the reverse inequality, let $G$ be an $n$-dimensional subspace
of the corresponding Sobolev space. If $G\cap\mathcal{N}^{B}\neq\{0\},$
then $\sup_{f\in G\setminus\{0\}}R_{\nu}(f)=+\infty.$ Otherwise,
the restriction of $\iota$ to $G$ is injective, and therefore $\dim(\iota G)=n.$
Moreover, $\iota G\subset\dom(\mathcal{E}^{B})$. By \prettyref{lem:Minimizer},
applied to $u=\iota f$, 
\[
\mathcal{E}^{B}(\iota f,\iota f)\leq\int_{\Omega}|\nabla f|^{2}\,\mathrm{d}\Lambda,\qquad f\in G.
\]
Applying \prettyref{eq:min-max_Pf} to $L=\iota G$, we obtain 
\[
\lambda^{B}_{n}\leq\sup_{u\in\iota G\setminus\{0\}}\frac{\mathcal{E}^{B}(u,u)}{\|u\|^{2}_{L^{2}(\Omega,\nu)}}\leq\sup_{f\in G\setminus\{0\}}R_{\nu}(f).
\]
Taking the infimum over all such $G$ proves the full Sobolev-space
min--max formula.
\end{proof}

\begin{proof}[Proof of \prettyref{lem:Neumann-zero-mode}]
Let $c$ be constant and let $g\in\mathcal{N}^{N}$. Then $\int_{\Omega}\nabla c\cdot\nabla g\d\Lambda=0$
and $\langle\iota c,\iota g\rangle_{L^{2}(\nu)}=0.$ Thus $\langle c,g\rangle_{\iota}=0,$
so $c\in H^{N}_{\nu}$. Since $\nabla c=0$, $\mathcal{E}^{N}(\iota^{N}c,\iota^{N}c)=0.$

Conversely, if $\mathcal{E}^{N}(u,u)=0$, then the reduced representative
$f=(\iota^{N})^{-1}u$ satisfies $\nabla f=0.$ Connectedness of $\Omega$
implies that $f$ is constant.
\end{proof}

\printbibliography

\end{document}